\documentclass[sn-mathphys-num]{sn-jnl}% Math and Physical Sciences Numbered Reference Style 
\usepackage{graphicx}%
\usepackage{multirow}%
\usepackage{amsmath,amssymb,amsfonts}%
\usepackage{amsthm}%
\usepackage{mathrsfs}%
\usepackage[title]{appendix}%
\usepackage{xcolor}%
\usepackage{textcomp}%
\usepackage{manyfoot}%
\usepackage{booktabs}%
\usepackage{algorithm}%
\usepackage{algorithmicx}%
\usepackage{algpseudocode}%
\usepackage{listings}%

\usepackage{stmaryrd}
\usepackage{mathrsfs}
\usepackage{enumerate}
\usepackage{hyperref}
\usepackage{tikz-cd}

\newtheorem{thm}{Theorem}[section]

\newtheorem{lemma}[thm]{Lemma}

\newtheorem{conj}[thm]{Conjecture}
\theoremstyle{definition}
\newtheorem{defn}[thm]{Definition}

\theoremstyle{remark}

\newtheoremstyle{named}{}{}{\itshape}{}{\bfseries}{.}{.5em}{#1 \thmnote{#3}}
\theoremstyle{named}
\newtheorem*{namedtheorem}{Theorem}

\newtheoremstyle{named}{}{}{\itshape}{}{\bfseries}{.}{.5em}{#1 \thmnote{#3}}
\theoremstyle{named}

\newtheoremstyle{named}{}{}{\itshape}{}{\bfseries}{.}{.5em}{#1 \thmnote{#3}}
\theoremstyle{named}

\newcommand{\Q}{\mathbb{Q}}
\newcommand{\Z}{\mathbb{Z}}
\newcommand{\N}{\mathbb{N}}

\newcommand{\C}{\mathbb{C}}

\newcommand{\F}{\mathbb{F}}
\newcommand{\Fbar}{\bar{\F}}
\newcommand{\calW}{\mathcal{W}}
\newcommand{\calC}{\mathcal{C}}
\newcommand{\calA}{\mathcal{A}}

\newcommand{\A}{\mathbb{A}}

\newcommand{\frakp}{\mathfrak{p}}
\newcommand{\frakq}{\mathfrak{q}}

\newcommand{\frakn}{\mathfrak{n}}

\newcommand{\frakN}{\mathfrak{N}}
\newcommand{\fraka}{\mathfrak{a}}

\newcommand{\scrL}{\mathscr{L}}
\newcommand{\boldx}{\mathbf{x}}

\newcommand{\calO}{\mathcal{O}}
\newcommand{\Sel}{\text{Sel}}

\newcommand{\cyc}{\text{cyc}}

\newcommand{\Tate}{\text{Tate}}

\newcommand{\Ig}{\text{Ig}}

\DeclareMathOperator{\Ima}{im}
\DeclareMathOperator{\Gal}{Gal}

\newcommand{\Pic}{\text{Pic}}
\newcommand{\Cl}[1]{\mathcal{C}\ell \left( #1 \right)}
\DeclareMathOperator{\lcm}{lcm}

\DeclareMathOperator{\Hom}{\text{Hom}}

\numberwithin{equation}{section}
\newcommand{\new}{\text{new}}
\newcommand{\alg}{\text{alg}}

\newcommand{\G}{\mathbf{G}}
\newcommand{\scrP}{\mathscr{P}}

\newcommand{\vbar}{\overline{v}}

\newcommand{\Qbar}{\overline{\mathbb{Q}}}
\newcommand{\calV}{\mathcal{V}}
\newcommand{\D}{\mathbf{D}}
\newcommand{\B}{\mathbf{B}}
\newcommand{\dR}{\text{dR}}
\newcommand{\et}{\text{\'et}}

\newcommand{\Fil}{\text{Fil}}

\newcommand{\Gr}{\text{Gr}}
\newcommand{\scrF}{\mathscr{F}}
\newcommand{\boldA}{\mathbf{A}}
\newcommand{\boldT}{\mathbf{T}}
\newcommand{\Char}{\text{char}}
\newcommand{\Spec}{\text{Spec}}
\newcommand{\can}{\text{can}}

\newcommand{\Lie}{\text{Lie}}
\newcommand{\Tw}{\text{Tw}}
\newcommand{\frakg}{\mathfrak{g}}
\newcommand{\cris}{\text{cris}}
\newcommand{\Sym}{\text{Sym}}

\newcommand{\Aut}{\text{Aut}}
\newcommand{\calM}{\mathcal{M}}
\newcommand{\GL}{\text{GL}}
\newcommand{\bbH}{\mathbb{H}}
\newcommand{\rec}{\text{rec}}
\newcommand{\frakX}{\mathfrak{X}}
\newcommand{\frakc}{\mathfrak{c}}
\newcommand{\anal}{\text{anal}}
\newcommand{\ur}{\text{ur}}

\begin{document}

\title[Congruent modular forms \& anticyclotomic Iwasawa theory]{Congruent modular forms and anticyclotomic Iwasawa theory}

%%=============================================================%%
%% GivenName	-> \fnm{Joergen W.}
%% Particle	-> \spfx{van der} -> surname prefix
%% FamilyName	-> \sur{Ploeg}
%% Suffix	-> \sfx{IV}
%% \author*[1,2]{\fnm{Joergen W.} \spfx{van der} \sur{Ploeg} 
%%  \sfx{IV}}\email{iauthor@gmail.com}
%%=============================================================%%

\author[1, 2]{\fnm{Dac-Nhan-Tam} \sur{Nguyen}}\email{tamnguyen@math.ubc.ca}

%\author[2,3]{\fnm{Second} \sur{Author}}\email{iiauthor@gmail.com}
%\equalcont{These authors contributed equally to this work.}

%\author[1,2]{\fnm{Third} \sur{Author}}\email{iiiauthor@gmail.com}
%\equalcont{These authors contributed equally to this work.}

\affil[1]{\orgdiv{Department of Mathematics}, \orgname{University of British Columbia}, \orgaddress{\street{1984 Mathematics Road}, \city{Vancouver}, \postcode{V6T 1Z2}, \state{British Columbia}, \country{Canada}}}
\affil[2]{ORCID: 0000-0002-3117-5028}

%\affil[2]{\orgdiv{Department}, \orgname{Organization}, \orgaddress{\street{Street}, \city{City}, \postcode{10587}, \state{State}, \country{Country}}}

%\affil[3]{\orgdiv{Department}, \orgname{Organization}, \orgaddress{\street{Street}, \city{City}, \postcode{610101}, \state{State}, \country{Country}}}

%%==================================%%
%% Sample for unstructured abstract %%
%%==================================%%

\abstract{Let $p$ be an odd prime. Consider normalized newforms $f_1, f_2$ such that both forms satisfy the Heegner hypothesis for an imaginary quadratic field $K$ and suppose that they induce isomorphic residual Galois representations. In the work of Greenberg-Vatsal \cite{GV00} and Emerton-Pollack-Weston \cite{EPW06}, the authors compare the cyclotomic Iwasawa $\mu$ and $\lambda$-invariants of $f_1$ and $f_2$. We extend this to the anticyclotomic indefinite setting by comparing the BDP $p$-adic $L$-functions attached to $f_1$ and $f_2$. Using this comparison, we obtain arithmetic implications for both generalized Heegner cycles and the Iwasawa main conjecture.}

\keywords{Anticyclotomic Iwasawa theory, congruent modular forms, $p$-adic $L$-functions, Heegner cycles.}

%%\pacs[JEL Classification]{D8, H51}

\pacs[MSC Classification]{11R23 (primary); 11G40 (secondary)}

\maketitle

\section*{Acknowledgements}
I would like to thank my advisors Antonio Lei and Sujatha Ramdorai for their support during the time I carried out this work. The project grew out of a problem suggested by Antonio Lei, and I am thankful for his continuous guidance and feedback. I would also like to thank Ashay Burungale for his interest in my work and for providing important feedback on previous drafts of this paper.

I would like to thank Devang Agarwal for a helpful discussion on quadratic twists that appeared in Lemma \ref{lem:cong}. Finally, I wish to thank Francesc Castella for answering some questions I had at the early stage of this project and for his work with Ming-Lun Hsieh \cite{CH18}.

\section{Introduction}

Let $f \in S_{2r}(\Gamma_0(N))^\new$ be a normalized newform of weight $2r$ and level $N$ that is an eigenform for all Hecke operators. Fix an odd prime $p \nmid N$ and let $F$ be a finite extension of $\Q_p$ containing the Fourier coefficients of $f$. Let $V_f$ be the  $p$-adic Galois representation attached to $f$ and let 
\[\rho_f: G_\Q \rightarrow \Aut(V_f(r))\] be its self-dual Artin twist. Here, $G_\Q$ denotes the absolute Galois group $\Gal(\overline{\Q}/\Q)$. Denote by $\overline{\rho}_f$ the associated semisimplified residual representation. 

Let $K/\Q$ be an imaginary quadratic field of discriminant $-D_K$ and let $p > 2$ be a rational prime that split in $K$ as $(p) = \frakp \overline{\frakp}$. Define the following hypothesis for $f \in S_{2r}(\Gamma_0(N))^{\new}$
\[\begin{cases} \label{Hyp}
	p \nmid 2 (2r - 1)! N \phi(N), \\
	\text{every prime } \ell \mid N \text{ is split in $K/\Q$}. \\ 
\end{cases} \tag{Heeg}\]

The second condition is known as the strong Heegner hypothesis. In such a setting, one may construct the Bertolini-Darmon-Prasanna (BDP) anticyclotomic $p$-adic $L$-function $\scrL_{\frakp}(f)$ attached to $f$ in the sense of \cite{BDP13, Bra11, CH18}. This paper closely follows the work of Castella-Hsieh \cite{CH18}, whose construction of the $p$-adic $L$-function originates from the work of Brako\v{c}evi\'{c} \cite{Bra11}. This $p$-adic $L$-function is defined as an element of the Iwasawa algebra $\calW \llbracket \Gamma_K^{-} \rrbracket$ where $\calW$ is a finite extension of the completed maximal unramified extension $\widehat{\Q_p^{nr}}$ over $\Q_p$ and $\Gamma_K^{-}$ is the Galois group of the anticyclotomic extension over $K$.

It is natural to ask how the Iwasawa $\mu$ and $\lambda$-invariants of $\scrL_{\frakp}(f_1)$ and $\scrL_{\frakp}(f_2)$ differ for newforms $f_1$ and $f_2$ whose residual representations are isomorphic. This type of question was first studied in \cite{GV00} over the cyclotomic extension, which was then generalized in \cite{EPW06}. The papers \cite{PW11}, \cite{Kim17}, \cite{CKL17} give analogous results in the definite anticyclotomic setting. 

In the indefinite anticyclotomic setting, congruences between the BDP $p$-adic $L$-functions have been studied in \cite{LMX23} for the weight $2$ case. In this setting, Kriz-Li studied the logarithms of Heegner points twisted by unramified characters which are interpolated by the BDP $p$-adic $L$-functions (see \cite[Theorem 3.9]{KL19}). The results in this paper can be seen as generalizations of \cite{LMX23, KL19} to forms of higher weights and generalized Heegner cycles. The techniques in this paper differ from \cite{LMX23} and the results are proved under fewer hypotheses. Moreover, for modular forms that are residually isomorphic with respect to an arbitrary prime power, we are able to show congruences between their $p$-adic $L$-functions with respect to the same prime power (see Theorem \ref{thm:main}). In a paper by Castella {\it et al.} \cite{CGLS22}, the authors use congruence methods to acquire new instances of the anticyclotomic Iwasawa main conjecture at Eisenstein primes. Their work can be seen as an extension of \cite[Theorem (1.3)]{GV00} to the BDP $p$-adic $L$-function whereas our work (in particular, Theorem \ref{thm:IMC}) extends \cite[Theorem (1.4)]{GV00}.

In \cite{KL19}, the authors study congruences by looking at the stabilizations of $f_1$ and $f_2$ at various primes $\ell$. These stabilizations are based on Hecke operators that act on classical modular form $f \in S_{2r}(\Gamma_0(N))$ via $f(q) \mapsto f(q^\ell)$. To study how the anticyclotomic $p$-adic $L$-function varies, this paper introduces some suitable moduli interpretations of these Hecke operators in the context of Igusa schemes in Section \ref{sec:V_l}, which will be relevant for the construction via Serre-Tate coordinates as defined in \cite{CH18, Bra11}. We also note that the moduli interpretations of some Hecke operators attached to the prime $p$ are discussed in \cite[Section 4.1.10]{Hid04}.
%TODO: Such interpretations are available for complex uniformization of Shimura varieites

We also explore arithmetic implications for Heegner cycles in Section \ref{sec:Heeg}, as well as the anticyclotomic Iwasawa main conjecture in Section \ref{sec:IMC}. We now state the main results of this paper. 

Suppose that $f_1 \in S_{2r_1}(\Gamma_0(N_1))^\new, f_2 \in S_{2r_2}(\Gamma_0(N_2))^\new$ are normalized Hecke eigenforms whose coefficients lie in $p$-adic field $F$. Suppose that the induced  semi-simplified $\mod{\varpi}^m$ Galois representations $\bar{\rho}_{f_1}, \bar{\rho}_{f_2}: G_\Q \rightarrow \GL_2(\calO_{F}/\varpi^{m}\calO_{F})$ are isomorphic, where $\varpi$ is the uniformizer of $\calO_{F}$. Let $\calW$ be the ring of integers of a finite extension of $\widehat{\Q_p^{nr}}$ containing $F$.

\begin{namedtheorem}[A (Theorem \ref{thm:main})]  
	Suppose that both $f_1, f_2$ satisfy hypothesis \eqref{Hyp} for $K/\Q$. One may write $(N_1) = \frak{N}_1 \overline{\frak{N}_1}$, $(N_2) = \frak{N}_2\overline{\frak{N}_2}$ as ideals in $\calO_K$. For each prime $\ell \mid N_1 N_2$, let $v \mid \frakN_1 \frakN_2$ be the corresponding prime above $\ell$. Then the following congruence holds:
	\begin{align*}
		\prod_{\ell \mid N_1 N_2} \scrP_{\vbar}(f_1) \scrL_{\frakp}(f_1) \equiv \prod_{\ell \mid N_1 N_2} \scrP_{\vbar}(f_2)  \scrL_{\frakp}(f_2)  \pmod{\varpi^m \calW \llbracket \Gamma_K^{-} \rrbracket},
	\end{align*}
	where $\scrP_{\vbar}(f_1)$ and $\scrP_{\vbar}(f_2)$ are defined in Definition \ref{defn:Euler}.
	Moreover, one has the following:
	\begin{enumerate}
		\item $\mu(\scrL_{\frakp}(f_1)) = 0$ if and only if $\mu(\scrL_{\frakp}(f_2)) = 0$.
		\item Assuming that $\mu(\scrL_{\frakp}(f_1)) = \mu(\scrL_{\frakp}(f_2)) = 0$, 
		\[\sum_{\ell \mid N_1 N_2} \lambda(\scrP_{\vbar}(f_1)) + \lambda(\scrL_{\frakp}(f_1)) = \sum_{\ell \mid N_1 N_2} \lambda(\scrP_{\vbar}(f_2)) +  \lambda(\scrL_{\frakp}(f_2)).\]
	\end{enumerate}
\end{namedtheorem}

% {\it Acknowledgments} The author would like to thank 
{\it Notation. }Throughout this paper, we fix embeddings $\iota_\infty: \overline{\Q} \hookrightarrow \C$ and $\iota_p: \overline{\Q} \hookrightarrow \C_p$. Let $v_p(\cdot)$ be the normalized additive valuation on $\C_p$ for which $v_p(p) = 1$. 

For each number field $L$, the embedding $\iota_p$ determines a choice of inclusion $L \subset \C_p$, or equivalently a prime in $L$ above $p$. We assume that this choice gives rise to the prime $\frakp$ in $K$ that is consistent with the splitting $p \calO_K = \frakp \overline{\frakp}$ given in the Introduction. We will denote by $L_p$ the completion of $L$ with respect to the prime induced by $\iota_p$. We will also denote by $\A_L$ the adeles of $L$ and $\widehat{L}$ the finite adeles. Moreover, let $L_\infty := \prod_{v \mid \infty} L_v$.

Let $K[c]$ be the ring class field of conductor $c$ over $K$, and write $K[p^\infty]$ for $\bigcup_{n \geq 0} K[p^n]$. Denote by $\tilde{\Gamma}$ the Galois group of $K[p^\infty]/K$, and let $\Gamma_K^{-}$ be the maximal pro-$p$ quotient so that $\Gamma_K^{-} = \Gal(K_\infty/K)$ is the Galois group of the anticyclotomic extension $K_\infty = \bigcup_{n \geq 0} K_n$ over $K$. Let $\rec_{\frakp}: \Q_p^\times = K_\frakp^\times \rightarrow \Gal(K^{ab}/K) \rightarrow \Gal(K[p^\infty]/K)$ be the local reciprocity map. We also write $K(\frakp^\infty)$ for the ray class field of conductor $\frakp^\infty$, and $K[c](\frakp^\infty)$ for the compositum of $K[c]$ and $K(\frakp^\infty)$.

\section{Geometric and $p$-adic modular forms} \label{modForm}

We follow the expositions in Brako\v{c}evi\'{c} \cite{Bra11} and Castella-Hsieh \cite{CH18} and recall the definitions of (geometric, $p$-adic) modular forms of levels $\Gamma_0(N)$ and $\Gamma_1(N)$. The main references for this section are \cite{Kat73}, \cite[Section 3]{Hid04}. 

Let $S$ denote a $\Z_{(p)}$-algebra and let $R$ denote some algebra over $S$. For an integer $N$, let $\mu_N$ be the group scheme of the $N$-th roots of unity and let $A[N]$ be the group scheme of the $N$-torsion points of an abelian variety $A$. 

Consider the isomorphism classes of triples $[(A, \eta_N, \omega)_{/R}]_{/\simeq}$, where $A/R$ is an elliptic curve and $\eta_N: \mu_N \rightarrow A[N]$ is the $\Gamma_1(N)$-level structure and $\omega \in H^0(A/R, \underline{\Omega}^1_{A/R})$ is a differential $1$-form. The functor classifying such triples is representable by an affine scheme $\calM_{\Gamma_1(N)}$ defined over $\Z[1/6N]$ \cite[Theorem 3.1]{Hid04}.

\begin{defn} (\cite[Section 3.2.3]{Hid04}
	\label{gamma-1-form} For each $S$-algebra $R$, consider the set of all triples $[(A, \eta_N, \omega)_{/R}] \in \calM_{\Gamma_1(N)}(R)$. A geometric modular form $f$ of weight $k$ and level $\Gamma_1(N)$ over $R$ is a rule assigning to such every triple $(A, \eta, \omega)_{/R}$ a value $f(A, \eta, \omega) \in R$ satisfying the following:	\begin{enumerate}
		\item $f(A, \eta, \omega) = f(A', \eta', \omega')$ if $(A, C, \omega) \simeq (A', C', \omega')$ over $R.$
		\item For any $S$-algebra homomorphism $\phi: R \rightarrow R'$, we have
		$$f((A, \eta, \omega) \otimes_R R') \simeq \phi(f(A, \eta, \omega))$$
		\item $f(A, \eta, \lambda \omega) = \lambda^{-k} f(A, \eta, \omega)$ for any $\lambda \in R^\times.$
		\item Let $\Tate(q)$ be the Tate curve $\G_m/q^{\Z}$ over $\Z(\!(q)\!)$, equipped with a level structure $\eta$ and a choice of differential $\omega$. Then $(\Tate(q), \eta, \omega))$ is defined over $S[\mu_d](\!(q^{1/d})\!)$ for some $d \mid N$, and we impose that $f(\Tate(q), \eta, \omega)) \in S[\mu_d] \llbracket q^{1/d} \rrbracket$ for every such $(\Tate(q), \eta, \omega).$
		
		Moreover, we say that $f$ is of level $\Gamma_0(N)$ if it also satisfies
		
		\item $f((A, \eta_N \circ b, \omega)_{/R}) = f(A, \eta_N, \omega)$ for any $b \in (\Z/N\Z)^\times$ with the canonical action of $(\Z/N\Z)$ on $\mu_N$ \cite[Section 3.1]{Bra11}.
		%a canonical level structure $\eta_{\can}: \mu_N \oplus \mu_{p^\infty} \hookrightarrow \G_m \rightarrow \G_m/q^{\Z}$ and a canonical differential $\omega_\can$ induced by $dt/t$ identifying $\G_m \simeq \Spec(\Z [t, t^{-1}])$ then the value $f(\Tate(q), \eta_{\can}, \omega_{\can}) \in S \llbracket q \rrbracket.$
	\end{enumerate}
\end{defn}

Define the $q$-expansion of $f$ as $f(\Tate(q), \eta_{\can}, du/u) \in S \llbracket q \rrbracket$, where $\eta_\can: \mu_N \oplus \mu_{p^\infty} \hookrightarrow \G_m \rightarrow \G_m/q^{\Z}$ is the canonical level structure, and $u$ is the canonical parameter of $\G_m = \Spec(\Z[u, u^{-1}]).$

To define $p$-adic modular forms, we first recall the Igusa scheme $\Ig(N)/\Z_{p}$, which is the moduli space parametrizing isomorphism classes of elliptic curves with $\Gamma_1(Np^\infty)$-structure. More precisely, for each $\Z_{p}$-algebra $R$, $\Ig(N)(R)$ is the isomorphism classes of tuples $(A, \eta)_{/R}$ where $A/R$ is an elliptic curve and $\eta = \eta_N \oplus \eta_p: \mu_N \oplus \mu_{p^\infty} \hookrightarrow A[N] \oplus A[p^\infty]$ is an immersion of group schemes (\cite[Section 3.2.7]{Hid04}).

\begin{defn} (\cite[Section 3.2.9]{Hid04}, \cite[Section 2.1]{CH18})
	Let $S$ be a $p$-adic ring and denote $S_m : = S/p^m S$. Define the space of $p$-adic modular forms of level $\Gamma_1(N)$ over $S$, denoted $V_{p}(\Gamma_1(N), S), $ as \[V_p(\Gamma_1(N), S) = H^0(\widehat{\Ig}(N), \calO_{\widehat{\Ig}(N)/S}) =  \varprojlim_{m} H^0(\Ig(N), \calO_{\Ig(N)/S_m}),\]
	where $\widehat{\Ig}(N)$ is the formal completion of $\Ig(N)$. In particular, $f$ is a function assigning to each $[(A, \eta)_{/R}] \in \Ig(N)(R)$ a value $f(A, \eta) \in R$, and they satisfy the following conditions:
	\begin{enumerate}
		\item $f((A, \eta)_{/R}) = f((A', \eta')_{/R})$ if $(A, \eta)_{/R} \simeq (A', \eta')_{/R}$.
		\item For any continuous homomorphisms of $S$-algebra $\phi: R \rightarrow R'$, we have
		\[f((A, \eta) \otimes_R R') \simeq \phi(f(A, \eta))\]
		\item For any level structure $\eta_N$ of type $\Gamma_1(N)$ on the Tate curve $\Tate(q)$, $f(\Tate(q), \eta_N \oplus \eta_p^{\can}) \in S\llbracket q^{1/N}\rrbracket$, where $\eta_p^{\can}$ is determined by the canonical image of $\zeta_p$ via $\boldmath{\G}_m \rightarrow \Tate(q).$
		
	\end{enumerate}
	A $p$-adic modular form is said to be of weight $k$ if $f(A, z^{-1} \eta_p, \eta_N) = z^k f(A, \eta_p, \eta_N)$ for all $z \in \Z_p^\times$. 
\end{defn}  

A geometric modular form gives rise to a $p$-adic modular form in the sense of \cite[(1.10.15)]{Kat78}: Let $R$ be a complete local $S$-algebra, and let $[(A, \eta)_{/R}] \in \Ig(N)(R)$. The $\Gamma_1(Np^\infty)$-level structure $\eta = \eta_N \oplus \eta_p$ determines a map $\widehat{\eta}_p: \widehat{\G}_m \xrightarrow{\sim} \widehat{A}$ \cite[(1.10.1)]{Kat78}(see also \cite[Proposition 1]{Tat67}). This in turn defines a differential $\omega(\widehat{\eta}_p): \Lie(A) \simeq \Lie(\widehat{A}) \rightarrow \Lie(\widehat{\G}_m) = R$. One can then define {\it the $p$-adic avatar} $\widehat{f}$ of $f$ (\cite{CH18}) by letting $\widehat{f}(A, \eta) = f(A, \eta, \omega(\widehat{\eta_p})).$
\section{CM points} 
This section follows \cite[Section 2]{CH18}.	\label{sec:CM}
Let $K$ be an imaginary quadratic field of discriminant $-D_K < 0$, and suppose that $p$ is split as $p = \frakp \overline{\frakp}$ in $\calO_K$. Let $f \in S_{2r}(\Gamma_0(N))^{\new}$ be a newform satisfying hypothesis \eqref{Hyp}. One may write $N = \frakN \overline{\frakN}$ for some ideal $\frakN$ in $\calO_K$. For a positive integer $c$, let $\calO_c := \Z + c \calO_K$ be the order of conductor $c$ in $K$, so that $\Gal(K[c]/K) \simeq \Cl{\calO_c}$.

For each prime-to-$\frakN\frakp$ integral ideal $\fraka$ of $\calO_c$, there is a CM point $x_\fraka = (A_\fraka, \eta_\fraka)$ as constructed in \cite[Section 2.3]{CH18} where $A_\fraka$ is the complex elliptic curve $\C/\fraka^{-1}$. Such a point is defined over a discrete valuation ring inside $\calV = \iota_p^{-1}(\calO_{\C_p}) \cap K^{ab}$. If $\fraka = \calO_c$, we write $(A_c, \eta_c)$ for $(A_{\calO_c}, \eta_{\calO_c})$. In this case, it is immediate that $A_\fraka = A_c/ A_c[\fraka]$ and the isogeny $\lambda_\fraka: A_c \rightarrow A_\fraka$ induced by the quotient map $\C/\calO_c \rightarrow \C/\fraka^{-1}$ yields $\eta_\fraka = \lambda_\fraka \circ \eta_c.$ An equivalent construction is also available in \cite[Section 5.1]{Bra11}. 

Let $\widehat{\Q_p^{nr}}$ be the $p$-adic completion of maximal unramified extension $\Q_p^{nr}$ of $\Q_p$. If $\fraka$ is a prime-to-$\frakN \frakp$ ideal of $\calO_c$ with $p \nmid c$, then $(A_\fraka, \eta_\fraka)$ has a model defined over $\calV^{nr}: = \calW \cap K^{ab}$ where $\calW$ is the ring of integers of $\widehat{\Q_p^{nr}}$. We will also denote this model by $(A_\fraka, \eta_\fraka)$ for the rest of this article. For $\calO_c = \calO_K$, wr																																																															ite $A$ for $A_{\calO_K}$ and fix a N\'{e}ron differential $\omega_A$ of $A$ over $\calV^{nr}$.

%TODO: maybe also give how the Galois action on these points is described by CM theory

If we let $\bbH$ be the complex upperhalf plane, then there is a complex uniformization 
\[Y_1(Np^n)(\C) = \GL_2(\Q)^{+} \backslash \bbH \times \GL_2(\widehat{\Q}) / U_1(Np^n)
\]
of complex points on the modular curve. Since the generic fiber $\Ig(N)_{/\Q}$ is given by
\[\Ig(N)_{/\Q} = \varprojlim_{n} Y_1(Np^n),\]
there is also a uniformization 
\begin{align*}
	\bbH\times \GL_2(\widehat{\Q}) & \rightarrow \Ig(N)(\C) \\ x = (\tau_x, g_x) & \mapsto (A_x, \eta_x)
\end{align*}
where $(A_x, \eta_x)$ is the corresponding moduli description. We refer readers to \cite[Section 2.1]{CH18} for the explicit form of this map. Moreover, we will also denote the right action of $\GL_2(\widehat{\Q})$ on $x = [(\tau_x, g_x)] \in \Ig(N)(\C)$ as
\[(\tau_x, g_x) * h := (\tau_x, g_x h).\]
Now, fix a choice of basis element $\vartheta$ for $\calO_K = \Z \oplus \Z \vartheta$.
Consider the embedding $K \hookrightarrow \GL_2(\Q)$ by the regular representation \cite[14]{Bra11}:
\[\rho(\alpha) \begin{pmatrix}
	\vartheta \\ 1
\end{pmatrix} = \begin{pmatrix}
	\alpha \vartheta \\ \alpha \end{pmatrix}.\]

For the choice of $\vartheta$ given in \cite[Section 2.3]{CH18}:
\[\vartheta = \frac{D' + \sqrt{-D_K}}{2}, \text{ where } D' = \begin{cases}
	D_K & \text{ if } 2 \nmid D_K \\
	D_K/2 & \text{ if } 2 \mid D_K
\end{cases},\]
the embedding $\rho: K \hookrightarrow \GL_2(\Q)$ is of the form
\[a + b \vartheta \mapsto \begin{pmatrix}
	a(\vartheta + \overline{\vartheta}) + b & - a \vartheta \overline{\vartheta} \\
	a & b
\end{pmatrix}.\]
Tensoring with $\A^{(\infty)}_{\Q}$ gives an embedding $\rho:K^\times \backslash \widehat{K}^\times \hookrightarrow \GL_2(\Q)\backslash \GL_2(\widehat{\Q}).$
Denote by $[\eta, g]$ the image of $(\eta, g)$ under the projection $\bbH\times \GL_2(\widehat{\Q}) \rightarrow \Ig(N)(\C)$. Then $[\eta, g] \in \Ig(N)(K^{ab}), $ and Shimura's reciprocity law states that
\[\rec_K(a)[(\vartheta, g)] = [\vartheta, \rho(\overline{a}) g]\]
where $\rec_K: K^\times \backslash \widehat{K}^\times \rightarrow \Gal(K^{ab}/K)$ is the geometrically normalized reciprocity law. We apply this to CM points as follows. Let $[(\vartheta, \xi_c)] \in \Ig(N)(\C)$ be the complex uniformization of the CM point $x_c := [(A_c, \eta_c)]$ for some $\xi_c \in \GL_2(\widehat{\Q})$. For an $\calO_c$-ideal $\fraka$ that is prime to $\frakN \frakp$, let $x_\fraka = (A_\fraka, \eta_a)$ and $a \in  \widehat{K}^{(cp)\times}$ be an idele such that $\fraka = a \widehat{\calO}_c \cap K$. Both $x_\fraka$ and $x_c$ are defined over $K[c](\frakp^\infty)$ and \[x_\fraka = [(A_\fraka, \eta_a)] = [\vartheta, \rho(\overline{a})^{-1} \zeta_c] = x_c^{\sigma_\fraka} \in \Ig(N)(K[c](\frakp^\infty))\] where $\sigma_\fraka = \rec_K(a^{-1})_{\mid K[c](\frakp^\infty)} \in \Gal(K[c](\frakp^\infty)/K)$, following Shimura's reciprocity law.
\section{Anticyclotomic $p$-adic $L$-functions} \label{sec:defn:L-func}

Let $f \in S_{2r}(\Gamma_0(N))^{\new}$ be a classical normalized eigenform, which we implictly assume to be an eigneform with respect to all Hecke operators unless otherwise stated. We will also denote by $f$ the associated geometric modular form, and let $f^{\flat}$ be the $p$-depleted geometric modular form with $q$-expansion $f^{\flat}(q) = \sum_{p \nmid n} a_n(f) q^n$ (\cite[10, 11]{CH18}).

\subsection{$t$-expansion of $p$-adic modular forms}
Recall that $\widehat{\Q_p^{nr}}$ is the completion of the maximal unramified extension of $\Q_p$ and $\calW$ is its ring of integers. 
Let $\Ig(N)_{/\calW}$ be the Igusa scheme over $\calW$, and let $\boldx = [(A_0, \eta)] \in \Ig(N)(\Fbar_p)$ where $A_0$ is an elliptic curve over ${\Fbar_p}$ and $\eta: \mu_{N} \oplus \mu_{p^\infty} \hookrightarrow A_0[N] \oplus A_0[p^\infty]$ is a $\Gamma_1(Np^\infty)$-level structure. Let $\hat{S_\boldx}$ be the local deformation space of $\boldx$, which represents the functor $$R \longmapsto \{\text{deformations of $A_0$ to $R$}\}$$ for Artin local rings $R$ with residue field $\Fbar_p$. Note that $\calW$ is the ring of Witt vectors of $\Fbar_p$ and $\hat{S}_\boldx$ is a $\calW$-scheme \cite[Section 3]{Kat81}.

One has a natural embedding $\hat{S}_\boldx \hookrightarrow \Ig(N)_{/\calW}$. By \cite[Theorem 2.1]{Kat81}, there is an equivalence of functors $$\hat{S_\boldx} \simeq \Hom_{\Z_p}(T_p(A_0)(\Fbar_p) \otimes T_p(A_0^t)(\Fbar_p)), \widehat{\G}_m),$$ where $A_0^t$ is the dual of $A_0$ and $T_p(A_0)$, $T_p(A_0^t)$ are, respectively, the Tate modules of $A_0$ and $A_0^t$. 

We denote by $q_{\calA}$ the pairing corresponding to the isomorphism class $[\calA_{/R}]$. As remarked in \cite[Section 3.1]{CH18}, $\eta_{p}$ determines a point $P_\boldx \in T_p(A_0^t)$ via the Weil pairing, which gives the canonical Serre-Tate coordinates $t: \hat{S_\boldx} \rightarrow \hat{\G}_m$ as $$t([\calA_{/R}]) = q_{\calA}(\lambda^{-1}_{\text{can}}(P_\boldx), P_\boldx),$$ together with an identification $\calO_{\hat{S}_\boldx} \simeq \calW \llbracket t - 1 \rrbracket$.
For a $p$-adic modular form $f \in V(N, \calW)$, we will denote $f(t) := f |_{\hat{S}_\boldx} \in \calW\llbracket t - 1 \rrbracket$.

Following \cite[Sections 3.3, 3.5]{Hid93}, we denote by $\text{Meas}(\Z_p; \calW)$ the space of $p$-adic measures with values in $\calW$. Recall the isomorphism 
$$\text{Meas}(\Z_p; \calW) \xrightarrow{\simeq} \calW \llbracket t - 1 \rrbracket$$
given by 
$$\varphi \mapsto \Phi_\varphi(t) = \sum_{n = 0}^\infty \left(\int_{\Z_p} {x \choose n} d \varphi(x) \right) (t - 1)^n = \int_{\Z_p} t^x d \varphi(x),$$
and let $df \in \text{Meas}(\Z_p; \calW)$ be the measure corresponding to $f$ under this isomorphism. 

Following the notation of \cite[p. 8]{CH18}, for a continuous function $\phi: \Z_p \rightarrow \calO_{\C_p}$, we define $(f \otimes \phi)(t) \in \calO_{\C_p} \llbracket t - 1 \rrbracket$ by
$$(f \otimes \phi)(t) = \int_{\Z_p} \phi(x)t^x df = \sum_{n \geq 0} \int_{\Z_p} \phi(x) {x \choose n} df(x) \cdot (t - 1)^n.$$

For a classical newform $f$ of weight $2r$ in $S_{2r}^\new(\Gamma_0(N))$, its Fourier coefficients $\{a_n(f)\}_{n > 0}$ lie in a $p$-adic field $F$. We may enlarge $\calW$ to be the ring of integers of the compositum $\widehat{\Q^{nr}_p} \cdot F$, so that both $\widehat{f}$ and $ \widehat{f^\flat}$ are $p$-adic modular forms over $\calW$. Note that $\calW$ is still a complete discrete valuation ring with residue field $\Fbar_p$ and define the $t$-expansions $\widehat{f}(t), \widehat{f^\flat}(t) \in \calW\llbracket t - 1 \rrbracket$ as above.

\subsection{Hecke characters}
A Hecke character $\chi: \A_K^\times /K^\times \rightarrow \C^\times$ is said to be of infinity type $(m, n)$ if $\chi(z_\infty) = z_\infty^m \overline{z}_\infty^n$. If $\chi$ has conductor $\frakc$, we will identify $\chi$ as a character on the ideal class group of conductor $\frakc$ via $\psi(\fraka) = \psi(a)$ where $a \in \A_K$ such that $a \widehat{\calO}_K \cap K = \fraka$, and $a_\frakq = 1$ for $\frakq \mid \frakc$. We write $\chi_\frakq$ for the $\frakq$-component of $\chi$. 

Moreover, we call $\chi$ an anticyclotomic Hecke character if $\chi$ is trivial on $\A_\Q^\times.$ For such a Hecke character $\chi$, the $p$-adic avatar $\widehat{\chi}: \widehat{K}^\times/K^\times \rightarrow \C_p^\times$ is defined by $\widehat{\chi}(a) = \iota_p \circ \iota_{\infty}^{-1}(\chi(a)) a_{\frakp}^{-m} a_{\bar{\frakp}}^{-n}.$ We also call a $p$-adic character $\rho: \widehat{K}^\times/K^\times \rightarrow \C_p^\times$ locally algebraic if $\rho = \widehat{\chi}$ for some complex Hecke character $\chi$, and define the conductor of $\rho$ to be the conductor of $\chi$. 

For every locally algebraic character $\rho: \widetilde{\Gamma} \rightarrow \calO_{\C_p}^\times,$ we denote by $\rho_\frakp$ the character $\rho_\frakp: \Q_p^\times \rightarrow \C_p^\times$ defined by $\rho_\frakp(\beta) = \rho(\rec_\frakp(\beta))$. For a general continuous function $\rho \in \calC(\widetilde{\Gamma}, \calO_{\C_p})$, we also define $\rho|[\fraka]: \Z_p^\times \rightarrow \calO_{\C_p}$ as $\rho|[\fraka](x) = \rho(\rec_\frakp(x) \rec_K(a)).$ Denote by $\frakX_{p^\infty} \subset \calC(\widetilde{\Gamma}, \calO_{\C_p})$ the set of locally algebraic $p$-adic characters $\widetilde{\Gamma} \rightarrow \calO^{\times}_{\C_p}$.

Finally, for a continuous local character $\phi: \Z_q^\times \rightarrow \C^\times$ that necessarily factors through $(\Z_q/q^n\Z_q)^\times$ for some $n$, we define its Gauss sum to be $\frakg(\phi) = \sum_{u \in (\Z/q^n \Z)^\times} \phi(u) \zeta^{u}, $ where $\zeta = e^{2\pi i/q^n}.$

\subsection{Anticyclotomic $p$-adic $L$-function}
For a positive integer $c = c_0 p^n$ where $\gcd(c_0, p) = 1$, let $\fraka$ be a fractional ideal of $\calO_c = \Z + c \calO_K$ and $[(A_\fraka, \eta_\fraka)] \in \Ig(N)(K[c])$ be the corresponding CM point on the Igusa scheme discussed in Section \ref{sec:CM}. Let $\widetilde{\Gamma}_K := \Gal(K[p^\infty]/K)$ be the Galois group of the compositum of ring class fields of $K$ with $p^{th}$-power conductor over $K$.

Following \cite[p.12]{CH18}, let $\fraka \subset \calO_{c_0}$ be a fractional ideal prime with $Np$ and let $t_\fraka$ be the canonical Serre-Tate coordinate of $\widehat{f^\flat}$ around the reduction $\boldx_\fraka = [(A_\fraka, \eta_\fraka)] \otimes_{\calW} \Fbar_p$ of $[(A_\fraka, \eta_\fraka)] \in \Ig(N)(K[c_0])$. Finally, set
$$\widehat{f_\fraka^\flat}(t_\fraka) := \widehat{f^\flat}(t_\fraka^{N(\fraka)^{-1} \sqrt{-D_K}^{-1}}) \in \calW \llbracket t_\fraka - 1 \rrbracket,$$
where $N(\fraka) = c^{-1} \# (\calO_{c_0}/\fraka)$ (\cite[Section 3.2]{CH18}).
%TODO: motivate the definition above

\begin{defn} \label{defn:L-func} \cite[Definition 3.7]{CH18}  Let $c_0 \geq 1$ be a positive integer such that $(c_0, pN) = 1$ and let $\psi$ be an anticyclotomic Hecke character of infinity type $(r,-r)$ of conductor $c_0 \calO_K$. Define  $\mathscr{L}_{\mathfrak{p}, \psi}(f)$ on $\widetilde{\Gamma}$ to be the $p$-adic measure on $\widetilde{\Gamma}$ given by
	\[\scrL_{\frakp, \psi}(f)(\rho)=\sum_{[\fraka] \in \Pic{\calO_{c_o}}} \psi(\fraka) \mathrm{N}(\fraka)^{-r} \cdot\int_{\Z_p^\times} \psi_\frakp \rho|[\fraka]  d{\widehat{f_\fraka^\flat}} 
	\]
	for every continuous function $\rho: \widetilde{\Gamma} \rightarrow \calO_{\C_p}.$
	We can also view $\mathscr{L}_{\mathfrak{p}, \psi}(f)$ as an element in the semi-local ring $\mathcal{W} \llbracket\widetilde{\Gamma}\rrbracket$. It is known that $\scrL_{\frakp, \psi}(f) \neq 0$ \cite[Theorem 3.9]{CH18}.
\end{defn}

For a character $\rho: \widetilde{\Gamma} \rightarrow \calO_{\C_p}^\times$, we define the map $\Tw_\rho: \calW \llbracket \widetilde{\Gamma} \rrbracket \rightarrow \calW \llbracket \widetilde{\Gamma} \rrbracket$ given by $\sigma \mapsto \rho(\sigma)\sigma$ for $\sigma \in \widetilde{\Gamma}$. We will denote $\scrL_{\frakp}(f) := \Tw_{\widehat{\psi}^{-1}}(\scrL_{\frakp, \psi}(f))$, which takes the value
\[
\scrL_{\frakp}(f)(\rho)=\sum_{[\fraka] \in \Pic{\calO_{c_0}}} \mathrm{N}(\fraka)^{-r} \cdot \int_{\Z_p^\times} \rho|[\fraka](x) x^{-r} d{\widehat{f_\fraka^\flat}}
\]
for every continuous function $\rho: \widetilde{\Gamma} \rightarrow \calO_{\C_p}$ (see also \cite[Definition 4.2]{BCK21}). For simplicity, we may assume that $c_0 = 1$ and $\Pic\calO_{c_0} = \Pic(\calO_K)$. 

\subsection{The $\theta$ operator} Let $\theta$ be the operator $t\frac{d}{dt}$ on $\calW \llbracket t - 1 \rrbracket$ and for $k < 0$ define
\[\theta^{k} := \lim_{m \rightarrow \infty} \theta^{k + (p - 1) p^{m}}.\]
To see that this is well-defined, see \cite[Section 4.5]{Bro15}. For $k > 0$ and $f(t) \in \calW \llbracket t - 1 \rrbracket$, it is well known (for example, via \cite[3.5(5)]{Hid93}) that
\[\theta^{k} f(t) = \int_{\Z_p^\times} t^{x} x^k df,\]
and the same identity also holds for $k < 0$. Thus we may re-write the definition of $\scrL_\frakp(f)$ as
\begin{align*}
	\scrL_{\frakp}(f)(\rho)= & \sum_{[\fraka] \in \Pic{\calO_{K}}} \mathrm{N}(\fraka)^{-r} 
	\cdot (\theta^{-r} \widehat{f}_\fraka^\flat \otimes \rho|[\fraka])(A_\fraka, \eta_\fraka) \\
	= & (\sqrt{-D_K})^{r} \sum_{[\fraka] \in \Pic \calO_K} ((\theta^{-r} \widehat{f}^{\flat})_\fraka \otimes \rho|[\fraka])(A_\fraka, \eta_\fraka) 
\end{align*}
for any continuous function $\rho: \widetilde{\Gamma} \rightarrow \calO_{\C_p}$.

\section{Congruent modular forms} \label{sec:cong}

Let $f_1\in S_{2r_1}^\new(\Gamma_0(N_1)), f_2 \in S_{2r_2}(\Gamma_0(N_2))^{\new}$ be normalized Hecke eigenforms of weight $2r_1$, $2r_2$ and levels $N_1, N_2$, respectively. Suppose that both $f_1$ and $f_2$ satisfy Hypothesis \eqref{Hyp}. Then there exist ideals $\frakN_1$, $\frakN_2$ in $\calO_K$ such that $N_1 \calO_K = \frakN_1 \overline{\frakN_1}$ and $N_2\calO_K = \frakN_2 \overline{\frakN_2}$. Further assume that for every $\ell \mid \gcd(N_1, N_2)$, one has $\gcd(\ell, \frakN_1) = \gcd(\ell, \frakN_2)$ so that $\calO_K/\lcm(\frakN_1, \frakN_2) \simeq \Z/\lcm(N_1, N_2)\Z$. 

%TODO: Congruent p-adic modular forms

We first show that $\scrL_\frakp(f_1)$ and $\scrL_\frakp(f_2)$ are congruent when their $q$-expansions are congruent. 

\begin{lemma} \label{lem:cong}
	Suppose that $f_1\in S_{2r_1}(\Gamma_0(N_1))^{\new}$, $f_2 \in S_{2r_2}(\Gamma_0(N_2))^{\new}$ have the same level $N_1 = N_2$. Let $F$ be a $p$-adic field containing $\Q(\{a_n(f_1), a_n(f_2)\}_{n > 0})$ and let $\varpi$ be a uniformizer of $\calO_{F}$. Suppose that $a_n(f_1) \equiv a_n(f_2) \pmod{\varpi^{m} \calO_{F}}$ for every $n$. Then there are congruences $f_1^{\flat} \equiv f_2^{\flat} \pmod{\varpi^{m} \calO_{F}}$, $\widehat{f_1^\flat} \equiv \widehat{f_2^\flat} \pmod{\varpi^{m} \calO_{F}}$ between $p$-adic modular forms, and \[\scrL_{\frakp}(f_1) \equiv \scrL_{\frakp}(f_2) \pmod{\varpi^m \calW \llbracket \Gamma_K^{-} \rrbracket}.\]
\end{lemma}

\begin{proof}
	The congruences between $p$-adic modular forms follow from the $q$-expansion principle \cite[Corollary 3.5]{Hid04}. We show that $\scrL_\frakp (f_1) (\rho) \equiv \scrL_\frakp (f_2)(\rho) \pmod{\varpi^m \calO_{\C_p}} $ for every continuous map $\rho: \widetilde{\Gamma} \rightarrow \calO_{\C_p}^\times$, and the congruence $\scrL_\frakp(f_1) \equiv \scrL_\frakp(f_2) \pmod{\varpi^m \calW \llbracket \Gamma_K^{-} \rrbracket}$ follows by the same argument as \cite[Theorem (1.10)]{Vat99}. Let $\chi_\cyc: G_\Q \rightarrow \Z_p^\times \subset \calO_{F}^\times$ be the cyclotomic character and let $\rho_{f_i}$ be the Weil-Deligne representation attached to $f_i$ for $i \in \{1, 2\}$. Since $\det(\rho_{f_i}) = \chi_\cyc^{2r_i - 1}$, we have the congruence
	
	\[\chi_\cyc^{2r_1 - 1} \equiv \chi_\cyc^{2r_2 - 1} \pmod{1 + \varpi^m \calO_{F}}. \]
	Suppose that $\varpi^m \calO_{F} \cap \Z_p = p^{m'} \Z_p$, then the congruence above actually holds in $(\Z_p/p^{m'} \Z_p)^\times \subset (\calO_{F}/\varpi^m \calO_{F})^\times:$
	\[\chi_\cyc^{2r_1 - 1} \equiv \chi_\cyc^{2r_2 - 1} \pmod{1 + p^{m'} \Z_p}.\]
	Hence, we have the congruence $2r_1 \equiv 2r_2 \pmod{\varphi(p^{m'})}$.
	
	Given a continuous function $\rho: \widetilde{\Gamma} \rightarrow \calO_{\C_p}$ , we may write
	\[
	\scrL_{\frakp}(f_1)(\rho) = \sum_{[\fraka] \in \Pic{\calO_{K}}} \mathrm{N}(\fraka)^{-r_1} 
	\cdot (\theta^{-r_1} \widehat{f}_{1, \fraka}^\flat \otimes \rho|[\fraka])(A_\fraka, \eta_\fraka)
	\]
	\[
	\scrL_{\frakp}(f_2)(\rho)=  \sum_{[\fraka] \in \Pic{\calO_{K}}} \mathrm{N}(\fraka)^{-r_2} 
	\cdot (\theta^{-r_2} \widehat{f}_{2, \fraka}^\flat \otimes \rho|[\fraka])(A_\fraka, \eta_\fraka)
	\]
	
	If $r_1 \equiv r_2 \pmod{\phi(p^{m'})}$, then  $n^{r_1} \equiv n^{r_2} \pmod{p^{m'}}$ for every $n \in \Z_p^\times$ and the result follows immediately. Otherwise, $n^{r_1} \equiv (\frac{n}{p}) n^{r_2} \pmod{p^{m'}}$ where $(\frac{\cdot}{p})$ is the Legendre symbol on $\F_p^\times$ defined as $(\frac{x}{p}) = x^{(p-1)/2}.$ With a slight abuse of notation, we will also denote by $(\frac{\cdot}{p})$ the lift of the Legendre symbol to $\Z_p^\times$. Since 
	$(\frac{\cdot}{p}) \otimes t^m = (\frac{m}{p}) t^m$ \cite[85]{Hid93} and $n^{r_1} \equiv (\frac{n}{p}) n^{r_2} \pmod{p^{m'}}$, we have the congruence
	\[\theta^{-r_1} \widehat{f}_{1, \fraka}^{\flat}(t) \equiv \left(\frac{\cdot}{p}\right) \otimes \theta^{-r_2} \widehat{f}_{2, \fraka}^{\flat}(t) \pmod{\varpi^m \calW\llbracket t - 1 \rrbracket}.\]
	Moreover, one also has $N(\fraka)^{-r_1} \equiv \left(\frac{N(\fraka)}{p}\right) N(\fraka)^{-r_2}$, from which it follows that
	\begin{equation}
		\label{eq:cong}
		\scrL_{\frakp}(f_1)(\rho) \equiv \sum_{[\fraka] \in \Pic{\calO_K}} N(\fraka)^{-r_2} \left(\frac{N(\fraka)}{p}\right) \rho(\fraka)  \left(\theta^{-r_2} \widehat{f_2^\flat} \otimes \left(\frac{\cdot}{p} \right) \rho_\frakp \right) (A_\fraka, \eta_\fraka) \pmod{\varpi^m \calW}.
	\end{equation}
	Define $\psi$ as the Hecke character such that $\psi(\fraka) = \left(\frac{N(\fraka)}{p}\right)$ for prime-to-$p$ fractional ideals $\fraka$ of $K$.
	%$\psi \A_K^\times /K^\times \rightarrow \C^\times$ such that $\psi_v$ is trivial for $v \nmid p$,
	Then $\psi$ is an anticyclotomic Hecke character of order $2$ and conductor $p$, and $\psi_\frakp: \calO_{K, \frakp}^{\times} \rightarrow \{\pm 1\}$ is the Legendre symbol $\left(\frac{\cdot}{p}\right)$. We may now rewrite the congruence \eqref{eq:cong} as
	\[
	\scrL_{\frakp}(f_1)(\rho) \equiv \sum_{[\fraka] \in \Pic{\calO_K}} N(\fraka)^{-r_2} (\psi \phi)(\fraka) (\theta^{-r_2} \widehat{f_2^\flat} \otimes \psi_\frakp \rho_\frakp)(A_\fraka, \eta_\fraka) \pmod{\varpi^m \calW}.
	\]
	In other words,
	\[\scrL_\frakp(f_1) \equiv \Tw_{\psi} \scrL_\frakp(f_2) \pmod{\varpi^m \calW \llbracket \widetilde{\Gamma} \rrbracket}.\]
	Since $\psi$ is a Hecke character of order $2$ and $p$ is odd, the restriction of $\psi$ to the anticyclotomic $\Z_p$-extension $\Gamma_K^{-}$ is trivial. Hence, one has the congruence
	\[\scrL_{\frakp}(f_1) \equiv \scrL_\frakp(f_2) \pmod{\varpi^m \calW \llbracket \Gamma_K^- \rrbracket}.\]
\end{proof}
%	\begin{rmk}
	%		We also remark that if $2r_1 \equiv 2r_2 \pmod{\# (\calO/\varpi^m \calO)^{\times}},$ then $f_1 \equiv f_2 \pmod{\varpi^m \calO_{L_p}}$ also as geometric modular forms if the coefficients of their $q$-expansions are congruent.
	%	\end{rmk}

\subsection{Hecke operators at $p$ in Serre-Tate coordinates} \label{sec:T_p}

Recall some Hecke operators in terms of the complex uniformization of Igusa schemes. Let $\fraka$ be a fractional ideal of $\calO_K$ and let $x_\fraka = [(A_\fraka, \eta_\fraka)] = [\vartheta, \rho(\overline{a})^{-1}\zeta_c]$ (see Section \ref{sec:CM}). For $z \in \Q_p$, we define $\frakn(z) := \begin{pmatrix}
	1 & z \\ 0 & 1
\end{pmatrix} \in \GL_2(\Q_p) \subset \GL_2(\widehat{\Q})$ and let $x_\fraka \ast \frakn(z) := [\vartheta, \rho(\overline{a})^{-1}\zeta_c \frakn(z)]$ under the action of $\GL_2(\widehat{\Q})$ on $\Ig(N)(\C)$. 

By \cite[Proposition 3.3]{CH18}, for a primitive Dirichlet character $\phi: (\Z/p^n \Z)^\times \rightarrow \calO_{\C_p}^\times$, the integral in Definition \ref{defn:L-func} can be written as
\[f_\fraka \otimes  \phi(x_\fraka) = p^{-n} \frakg(\phi) \sum_{u \in (\Z/p^n\Z)^\times} \phi^{-1}(u) f(x_{\fraka} \ast \frakn(up^{-n})).\]

In \cite[Proposition 6.4]{Bra11}, the author discusses the moduli interpretations of $x_\fraka \ast \frakn(up^{-n})$ for $u \in (\Z/p^n\Z)^\times$ as quotients of $A_\fraka$ by certain rank-$p^n$ subgroup schemes of $A_\fraka [p^\infty]$. Moreover we have $x_\fraka \ast \frakn(up^{-n}) \otimes \Fbar_p = x_\fraka \otimes \Fbar_p$, and the Serre-Tate coordinate of $x_\fraka \ast \frakn(up^{-n})$ is given by \[t_\fraka(x_\fraka \ast \frakn(up^{-n})) = \zeta_{p^n}^{-u \N(\fraka)^{-1} \sqrt{-D_K}^{-1}}\] according to \cite[Lemma 3.2]{CH18}.
\subsection{Hecke operators at $\ell \neq p$ in Serre-Tate coordinates} \label{sec:V_l}

%TODO: maybe just include a thorough discussion on some Hecke operators, may not even need to mention complex uniformization.

Let $f \in S_{2r}(\Gamma_0(N))^\new$ be a normalized newform of weight $2r$ and level $N$ that is an eigenform for all Hecke operators.

Let $\ell \neq p$ be a rational prime. For $\lcm(N, \ell) \mid N^\sharp$, one may naturally identify $f$ as a form of level $N^\sharp$. For an ordinary test triplet $(A, \eta_{N^\sharp}, \omega) \in \calM_{\Gamma_1(N^\sharp)}$ with level $N^\sharp$ structure $\eta_{N^{\sharp}}$, let $C \subset A[N^\sharp]$ be the image of the level structure $\eta_{N^\sharp}$. 

%	and for any $N \mid N^\sharp$, denoted by $\eta_{N}$ the restriction of $\eta_{N^\sharp}$ to $\mu_{N} \subset \mu_{N^\sharp}.$ 

Let $\pi$ be the projection $A \rightarrow A/C[\ell]$. Note that the morphism \[\pi \circ \eta_{N^\sharp}: \mu_{N^\sharp} \rightarrow C/C[\ell]\] has kernel $\mu_\ell$, and denote by $\overline{\pi \circ \eta_{N^\sharp}}$ the isomorphism
\[\overline{\pi \circ \eta_{N^\sharp}}: \mu_{N^\sharp}/\mu_{\ell} \rightarrow C/C[\ell].\]
Moreover, denote by $(\cdot)^{1/\ell}$ the inverse of the isomorphism $\mu_{N^\sharp}/\mu_{\ell} \xrightarrow{\zeta \rightarrow \zeta^{\ell}} \mu_{N^\sharp \ell^{-1}}$.

\begin{defn} \label{defn:Vell}
	Define the 'dividing by $\ell$-level structure' operator $V_\ell$ on ordinary test triplets as
	\[V_\ell(A, \eta_{N^\sharp}, \omega) = (A/C[\ell], \pi \circ \eta_{N^\sharp} \circ (\cdot)^{1/\ell}, \check{\pi}^* \omega),\]
	where $\pi: A \rightarrow A/C[\ell]$ is the canonical projection and $\check{\pi}: A/C[\ell] \rightarrow A$ is its dual isogeny. Here,  $\check{\pi}^*$ is the map on differential forms induced by $\check{\pi}$.
	
	%	This is well-defined because $\check{\pi}: \pi(C) \rightarrow A[N^\sharp \ell^{-1}]$ is an immersion of group schemes: suppose that $\check{\pi}(P) = O$, where $P = \pi (Q)$ for some $Q \in C$. Note that $\ell Q = \check{\pi} \circ \pi (Q) = O$, hence $Q \in A[\ell] \cap C = C[\ell]$, hence $\pi(Q) = O$.
	
	The map $V_\ell$ induces an operator $V_\ell^*$ on the space of classical modular forms of level $\Gamma_0(N^\sharp)$ via the rule $V_\ell^* f (A, \eta_{N^\sharp}, \omega) = f (V_\ell (A, \eta_{N^\sharp}, \omega))$, which acts on $q$-expansions as $f(q) \mapsto f(q^\ell)$ \cite[14, 15]{KL19}. 
\end{defn}

\begin{defn}
	Define the $(\ell)$-stabilization for a newform $f$ of conductor $N$ and weight $2r$ as:
	\[f^{(\ell)} = \begin{cases}
		f - a_\ell(f)V_\ell^\ast f + \ell^{2r - 1} V_\ell^\ast V_\ell^\ast f & \text{ if } \ell \nmid N, \\
		f - a_\ell(f)V_\ell^\ast & \text{ otherwise.} 
	\end{cases}\]
	where $f$ is viewed as a form of level $N^\sharp$.
\end{defn}

We now give a description of the Hecke operators above for a $p$-adic modular form $f \in V_p(N, \calW)$ of level $N$. Suppose that $N \mid N^\sharp$ and let $(A, \eta_{N^\sharp} \times \eta_p) \in \Ig(N^\sharp)$. 
There is a natural map \[\frac{N^\sharp}{N}: \Ig(N^\sharp) \rightarrow \Ig(N)\]
\[(A, \eta_{N^\sharp} \times \eta_{p}) \mapsto (A, \eta_{N} \times \eta_p),\] where $\eta_N$ is the restriction of $\eta_{N^\sharp}$ to $\mu_N$.

This induces an identification of $p$-adic modular forms of level $N$ as forms of level $N^\sharp$:
\[[N^\sharp/N]^\ast: V_p(N, \calW) \hookrightarrow V_p(N^\sharp, \calW).\]

\begin{defn} \label{defn:Vell-padic}
	For $\lcm(N, \ell) \mid N^\sharp$, define the following analogue of the $V_\ell$ operator for $p$-adic modular forms:
	\[V_\ell: \Ig(N^\sharp) \rightarrow \Ig(N^{\sharp}\ell^{-1})\]
	\[(A, \eta) \mapsto (A/C[\ell], \overline{\pi \circ \eta_{N^\sharp}} \circ (\cdot)^{1/\ell} \times \check{\pi}^{-1} \circ \eta_p)\]
	for $C[\ell] := \Ima(\eta)$, and similarly define $V_\ell^\ast f(A, \eta) = f(V_\ell(A, \eta))$ for a $p$-adic modular form $f$ of level $N^\sharp$. We also note that $\check{\pi}^{-1} \circ \eta_p = \frac{1}{\ell} \circ \pi \circ \eta_p$.
\end{defn}

For a complete local $\calW$-algebra $R$ and $[(A, \eta)_{/R}] \in \Ig(N)(R)$, recall from Section \ref{modForm} that the $\Gamma_1(Np^\infty)$-level structure $\eta = \eta_N \oplus \eta_p$ determines a map $\widehat{\eta}_p: \widehat{\G}_m \xrightarrow{\sim} \widehat{A}$ \cite[(1.10.1)]{Kat78}(see also \cite[Proposition 1]{Tat67}), which defines a differential $\omega(\widehat{\eta}_p): \Lie(A) \simeq \Lie(\widehat{A}) \rightarrow \Lie(\widehat{\G}_m) = R$. A geometric modular form $f$ can then be identified as a $p$-adic modular form via the rule $\widehat{f}(A, \eta)  = f(A, \eta, \omega(\widehat{\eta_p}))$. To show the compatibility of the $V_\ell$ operator defined on geometric modular forms and $p$-adic modular forms, we begin with the following

\begin{lemma}
	Let $\phi:A/R \longrightarrow A'/R$ be an isogeny of elliptic curves. Suppose that $\eta_p$ is a $p^\infty$-level structure on $A/R$, and $\phi \circ \eta_p$ is the $p^\infty$- level structure on $A'/R$ induced by $\phi$. Then $(\phi^\ast)^{-1} \omega(\widehat{\eta_p}) = \omega(\widehat{\phi \circ \eta_p})$, where the map $\phi^\ast: H^0(A'/R, \underline{\Omega}_{A'/R}^1) \rightarrow H^0(A/R, \underline{\Omega}_{A/R}^1)$ between differential $1$-forms is induced by $\phi$.
\end{lemma}

\begin{proof}
	Throughout this proof, we use the equivalence between the category of divisible commutative Lie groups and the category of connected $p$-divisible groups	 \cite[Proposition 1]{Tat67}.
	
	Let $\phi: A/R \longrightarrow A'/R$ be an isogeny. Then there are induced maps
	\[\widehat{\G}_m \xrightarrow{\widehat{\eta_p}} \widehat{A} \xrightarrow{\widehat{\phi}} \widehat{A}', \]
	\[\Lie(\widehat{\G}_m) \xrightarrow{\Lie(\widehat{\eta}_p)} \Lie(\widehat{A}) \xrightarrow{\Lie(\widehat{\phi})} \Lie(\widehat{A}').\]
	
	Recall that $\omega(\widehat{\eta_p})$ (respectively $\omega(\widehat{\phi \circ \eta_p})$) is defined as the inverse of $\Lie(\widehat{\eta}_p)$ (respectively $\Lie(\widehat{\phi \circ \eta_p})$): 
	\[\omega(\widehat{\eta}_p): \Lie(A) \simeq \Lie(\widehat{A}) \xrightarrow{\Lie(\widehat{\eta_p})^{-1}} \Lie(\widehat{\G}_m) = R, \]
	\[\omega(\widehat{\phi \circ \eta_p}): \Lie(A') \simeq \Lie(\widehat{A'}) \xrightarrow{\Lie(\widehat{\phi \circ \eta_p})^{-1}} \Lie(\widehat{\G}_m) = R.\]	
	
	Hence, we have $\Lie(\widehat{\phi})^\ast \omega(\widehat{\phi \circ \eta_p}) = \omega(\widehat{\eta}_p)$ by functoriality, where $\Lie(\widehat{\phi})^\ast$ is the pull-back map induced by $\Lie(\widehat{\phi})$. Moreover, the map $\Lie(\widehat{\phi})^\ast$ is the same as $\phi^\ast: H^0(A'/R, \underline{\Omega}_{A'/R}^1) \rightarrow H^0(A/R, \underline{\Omega}_{A/R}^1)$ by definition, and we have $\phi^\ast \omega(\widehat{\phi \circ \eta_p}) = \omega(\widehat{\eta}_p).$
\end{proof}

\begin{lemma}
	Let $\widehat{f} \in V_p(N, \calW)$ be the $p$-adic avatar of a geometric modular form $f$. Then \[V_\ell^\ast \widehat{f} = \widehat{V_\ell^\ast f},\]
	where the $V_\ell^{\ast}$ operator on the left-hand side acts on $p$-adic modular forms (Definition \ref{defn:Vell-padic}) and the $V_\ell^{\ast}$ operator on the right-hand side acts on geometric modular forms (Definition \ref{defn:Vell}).
\end{lemma}

\begin{proof}
	This follows directly from $\check{\pi}^\ast \omega(\widehat{\eta_p}) = \omega(\widehat{\check{\pi}^{-1} \circ \eta_p})$ by the previous lemma, and the definitions
	\[\widehat{V_\ell^\ast f}(A, \eta) = f(A/C[\ell], \overline{\pi \circ \eta_{N^\sharp}} \circ (\cdot)^{1/\ell}, \check{\pi}^\ast \omega(\eta_p)),\]
	\[
	V_\ell^\ast \widehat{f}(A, \eta) = f(A/C[\ell], \overline{\pi \circ \eta_{N^\sharp}} \circ (\cdot)^{1/\ell}, \omega(\check{\pi}^{-1} \circ \eta_p)).\] \end{proof}

%TODO: This is partly mentioned in Brakocevic

Let $\ell$ be a prime that splits in $\calO_K$ as $\ell = v \vbar$, and let $N^\sharp$ be such that $\lcm(N, \ell^2) \mid N^\sharp$. For every fractional ideal $\fraka$ of $\calO_K$ and every level $M$ divisible by $N$, let $x_{\fraka} = (A_{\fraka}, \eta_\fraka) \in \Ig(M)$ be a CM point satisfying $\Ima(\eta_\fraka)[\ell^\infty] = \Ima(\eta_\fraka)[\ell^\infty] \cap A[v^\infty].$ We assume that these points are compatible with the projections $\Ig(M') \rightarrow \Ig(M)$ for $M \mid M'$. It follows from definitions that the value of a $p$-adic modular form $f \in V_p(N, \calW)$ at such a CM point does not depend on the implicit level under the natural identification $V_p(N, \calW) \hookrightarrow V_p(M, \calW)$ for any $M$ divisible by $N$. 

\begin{lemma} \label{lem:Vell-1}
	Let $x_{\fraka} = (A_\fraka, \eta_\fraka) \in \Ig(N^\sharp)$ be a CM point. Then $V_\ell(x_\fraka) = (A_{\vbar^{-1}\fraka}, \eta_{\vbar^{-1} \fraka}) \in \Ig(N^\sharp \ell^{-1})$. As a consequence, we have \[V_\ell^\ast f(x_\fraka) = f(x_{\vbar^{-1}\fraka})\]
	for a $p$-adic modular form $f \in V_p(N, \calW).$
	
	%		where \[x_{\vbar^{-1} \fraka} = (A_{\vbar^{-1}\fraka}, \eta_{\vbar^{-1} \fraka \mid A_{\vbar^{-1} \fraka}[\frakN^\sharp \vbar^{-1}]}) \in \Ig(N^{\sharp} \ell^{-1}).\]
\end{lemma}

\begin{proof} For ease of notation, we will denote $A = A_\fraka$ and $\eta = \eta_\fraka$ in this proof. Denote by $\pi_v$ the projection $A \rightarrow A/A[v]$, and by $\pi_{\vbar}$ the projection $A_{v} \rightarrow A_v/A_v[\vbar]$ where $A_v = A/A[v]$. Observe that
	\begin{align*}
		\vbar \star V_\ell(x_\fraka) & = \vbar \star (A/A[v], \overline{\pi_v \circ \eta_{N^\sharp}} \circ (\cdot)^{1/\ell} \times \check{\pi}_v^{-1} \circ \eta_p) \\ & = (A/A[\ell], \pi_{\vbar} \circ \overline{\pi_v \circ \eta_{N^\sharp}} \circ (\cdot)^{1/\ell} \times \pi_{\vbar} \circ \check{\pi}_v^{-1} \circ \eta_p)).
	\end{align*}
	
	We claim that the isomorphism \[\iota: A/A[\ell] \rightarrow A\]
	\[x + A[\ell]\mapsto [\ell]x\] introduced in Lemma 3.5 of \cite{KL19} gives rise to the isomorphism between the tuples
	\[(A/A[\ell], \pi_{\vbar} \circ \overline{\pi_v \circ \eta_{N^\sharp}} \circ (\cdot)^{1/\ell} \times \pi_{\vbar} \circ \check{\pi}_v^{-1} \circ \eta_p)) \simeq (A, (\eta_{N^\sharp/\ell} \times \eta_p)),\]
	where $\eta_{N^\sharp \ell^{-1}}$ is the restriction of $\eta_{N^\sharp}$ to $\mu_{N^\sharp \ell^{-1}}.$ 
	
	Indeed, following the argument in {\it loc. cit.}, the composition $\iota \circ \pi_{\vbar} \circ \pi_v$ is the multiplication by $\ell$ map $[\ell]: A \rightarrow A$. This implies that the dual isogeny $\check{\pi}_v$ of $\pi_v$ is  $\iota \circ {{\pi}_{\vbar}}$, so that $\iota \circ \pi_{\vbar} \circ \check{\pi}_v^{-1} \circ \eta_p = \eta_p.$ 
	
	Next, we show that $\iota \circ \pi_{\vbar} \circ \overline{\pi_v \circ \eta_{N^\sharp}} \circ (\cdot)^{1/\ell} = \eta_{N^\sharp \ell^{-1}}.$ Since $\iota \circ \pi_{\vbar} \circ \pi_{v}$ is just the multiplication by $\ell$ map, the composition $\iota \circ \pi_{\vbar} \circ \pi_v \circ \eta_{N^\sharp}$ is simply
	\[\mu_{N^\sharp} \xrightarrow{\eta_{N^\sharp}} A[\frakN^\sharp] \xrightarrow{[\ell]} A[\frakN^\sharp v^{-1}].\]
	
	The following diagram commutes:
	\begin{center}
		\begin{tikzcd}
			\mu_{N^\sharp} \arrow[r, "\eta_{N^\sharp}"] \arrow[d, "(\cdot)^{\ell}"] & A[\frakN^{\sharp}] \arrow[d, "\ell"] \\
			\mu_{N^\sharp \ell^{-1}} \arrow[r, "\eta_{N^\sharp \ell^{-1}}"] & A[\frakN^{\sharp} v^{-1}],
		\end{tikzcd}
	\end{center}
	which shows that $\iota \circ \pi_{\vbar} \circ \overline{\pi_v \circ \eta_{N^\sharp}} \circ (\cdot)^{1/\ell} = \eta_{N^\sharp \ell^{-1}}.$
\end{proof}

If $x \otimes \Fbar_p = x_\fraka \otimes \Fbar_p$, then the reduction $V_\ell(x) \otimes \Fbar_p$ of $V_\ell(x)$ is $x_{\vbar^{-1} \fraka} \otimes \Fbar_p$. Analogous to \cite[Lemma 4.8]{Bro15}, the relationship between their $t$-expansions is given by:
\[t_{\vbar^{-1}\fraka}(V_\ell(x)) = t_\fraka(x)^{\ell}.\]

It also follows from this identity that $V_\ell(x_{\fraka} \ast \frakn (up^{-n})) = x_{\vbar^{-1}\fraka} \ast  \frakn(up^{-n}).$ Indeed, 
\begin{align*}
	t_{\vbar^{-1} \fraka}(V_\ell(x_{\fraka} \ast \frakn (up^{-n}))) & = t_{\fraka}(x_{\fraka} \ast \frakn (up^{-n}))^{\ell} \\ & = \zeta_{p^n}^{-u \ell {\N(\fraka)}^{-1} \sqrt{-D_K}^{-1}} \\ & = \zeta_{p^n}^{-u {\N(\vbar^{-1} \fraka)}^{-1} \sqrt{-D_K}^{-1}} \\ & = t_{\vbar^{-1} \fraka}(x_{\vbar^{-1}\fraka} \ast  \frakn(up^{-n})).
\end{align*}
\begin{lemma} \label{lem:Vell-2}
	Let $\phi: \Z_p^\times \rightarrow \calO_{\C_p^\times}$ be a $p$-adic character of conductor $p^n$. We have the following identity:
	\[((\theta^{-r} V_\ell^{\ast}f)_\fraka \otimes \phi)(x_{\fraka}) = \ell^{-r} (  (\theta^{-r} f)_{\vbar^{-1} \fraka} \otimes \phi)(x_{\vbar^{-1} \fraka}).\]
\end{lemma}

\begin{proof}
	By examining $t$-expansions, observe that
	\[\theta^{-r} V_\ell^\ast f = \ell^{-r} V_\ell^\ast \theta^{-r} f\]
	for a $p$-adic modular form $f \in V_p(N, \calW)$. 
	Combined with Lemma \ref{lem:Vell-1}, we have:
	\begin{align*}
		((\theta^{-r} V_\ell^{\ast}f)_\fraka \otimes \phi)(x_{\fraka}) & = p^{-n} \frakg(\phi) \cdot \sum_{u \in (\Z/p^n\Z)^\times} \phi^{-1}(u) (\theta^{-r} V_\ell^\ast f)(x_\fraka \ast \frakn(up^{-n})) \\  & = p^{-n} \frakg(\phi) \ell^{-r} \cdot \sum_{u \in (\Z/p^n\Z)^\times} \phi^{-1}(u)  (V_\ell^\ast \theta^{-r} f)(x_\fraka \ast \frakn(up^{-n})) \\ & = p^{-n} \frakg(\phi) \ell^{-r} \cdot \sum_{u \in (\Z/p^n\Z)^\times} \phi^{-1}(u) (\theta^{-r} f)(x_{\vbar^{-1} \fraka} \ast \frakn(up^{-n})) \\ & = \ell^{-r} ((\theta^{-r} f)_{\vbar^{-1} \fraka} \otimes \phi)(x_{\vbar^{-1} \fraka}).
	\end{align*}
\end{proof}

\begin{defn} \label{defn:Euler}
	Following \cite[8, 9]{GV00}, define $\scrP_v \in \calW \llbracket \Gamma_K^{-} \rrbracket$ such that \[\scrP_v(f) = \begin{cases}
		1 - a_\ell(f) \ell^{-r} \cdot \gamma_v + \ell^{-1} \cdot \gamma_v^2 \in \calW \llbracket \Gamma_K^{-} \rrbracket & \text{ if } \ell \nmid N, \\
		1 - a_\ell(f) \ell^{-r} \cdot \gamma_v & \text{ if } \ell \mid N.
	\end{cases}\]		
	where $\gamma_v \in \Gamma_K^{-}$ is the Frobenius at $v$. Define $\scrP_{\overline{v}}(f) \in \calW \llbracket \tilde{\Gamma} \rrbracket$ similarly.
\end{defn}	
Fix a topological generator $\gamma_0$ of $\Gamma_K^{-}$, and let $\calW \llbracket \Gamma_K^{-} \rrbracket \simeq \calW \llbracket T \rrbracket$ be the isomorphism given by $\gamma_0 \mapsto T + 1$.

\begin{lemma} \label{lem:mu=0}
	As elements in $\calW \llbracket \Gamma_K^{-} \rrbracket$, both $\scrP_v(f)$ and $\scrP_{\vbar}(f)$ have $\mu$-invariants $0$.
\end{lemma}

\begin{proof}
	One may write $\gamma_v = \gamma_0^{a}$ where $a \in \Z_p$. For $\ell \mid N$, $\scrP_{\overline{v}} = 1 - a_\ell(f) \ell^{-r} \cdot (1 + T)^{a}$. Let $a = \sum_{n \geq k} a_k p^k$, where $a_n \in \{0, \cdots, p - 1\}$ and $k$ is the smallest index such that $a_k \neq 0$. One has the following congruence: \[(1 + T)^{f_v} \equiv \prod_{n \geq k} (1 + T^{p^n})^{a_n} \pmod{\varpi},\]
	from which it follows that 
	\[\scrP_v(f)(T) \equiv 1 - a_\ell(f) \ell^{-r} (1 + T^{p^k})^{a_k} \equiv 1 - a_\ell(f) \ell^{-r} (1 + a_k T^{p^k}) \\ \pmod{(\varpi, T^{2p^{k}})},\]
	and therefore $\scrP_v(f)(T) \neq 0 \pmod \varpi$. The analogous statement for $\scrP_{\overline{v}}(f)$ also holds. 
	
	We can similarly show that $\mu(\scrP_\ell(f)) = 0$ for $\ell \nmid N$. Indeed, we may write $\scrP_v(f) = (1 - a_\ell \cdot \gamma_v)(1 - b_\ell \cdot \gamma_v)$, and it can be shown by the same argument as above that both $1 - a_\ell \cdot \gamma_v, 1 - b_\ell \cdot \gamma_v$ have $\mu$-invariants $0$.  The same argument applies to $\scrP_{\overline{v}}(f)$. 
\end{proof}

\begin{thm} \label{l-dep-2}
	Let $f^{(\ell)}$ be the $\ell$-depletion of $f$, considered as a geometric modular form of level $N^\sharp$ where $\lcm(\ell, N) \mid N^\sharp$. Then
	$\scrL_{\frakp}(f^{(\ell)}) = \scrP_{\vbar}(f) \scrL_{\frakp}(f).$
\end{thm}
\begin{proof} 
	For every locally algebraic character $\rho \in \frakX_{p^\infty}$, we use Lemma \ref{lem:Vell-2} to obtain the following identity:
	\begin{align*}
		\scrL_{\frakp}(V^\ast_\ell f)(\rho)& = (\sqrt{-D_K})^{r} \sum_{[\fraka] \in \Pic \calO_K} ((\theta^{-r} V_\ell^\ast \widehat{f}^{\flat})_\fraka \otimes \rho|[\fraka])(A_\fraka, \eta_\fraka) \\  & = \rho(\vbar) \ell^{-r} (\sqrt{-D_K})^{r} \sum_{[\fraka] \in \Pic{\calO_K}}  ((\theta^{-r} \widehat{f}^{\flat})_{\vbar^{-1}\fraka} \otimes \rho |[\vbar^{-1}\fraka]) (A_{\vbar^{-1}\fraka}, \eta_{\vbar^{-1}\fraka}) \\ & = \rho(\vbar) \ell^{-r} \scrL_{\frakp}(f)(\rho).
	\end{align*}
	Hence $\scrL_{\frakp}(f^{(\ell)}) = \scrL_{\frakp}(f - a_\ell(f) V_\ell^\ast f)(\rho) = (1 - a_\ell(f)\rho(\vbar)\ell^{-r})\scrL_{\frakp}(f)$ for $\ell \mid N$, and for $\ell \nmid N$ we have $\scrL_{\frakp}(f^{(\ell)}) = (1 - a_\ell(f)\rho(\vbar)\ell^{-r} + \rho(\vbar)^2 \ell^{-1}) \scrL_{\frakp}(f).$
\end{proof}

\begin{thm} \label{thm:main}
	Suppose that $f_1 \in S_{2r_1}(\Gamma_0(N_1))^\new, f_2 \in S_{2r_2}(\Gamma_0(N_2))^\new$ are newforms satisfying Hypothesis \eqref{Hyp} whose coefficients lie in some $p$-adic field $F$. Assume that $\calW$ is the ring of integers of a finite extension of $\widehat{\Q_p^{nr}} \cdot F$.

	Suppose that the induced  semi-simplified $\mod{\varpi}^m$ Galois representations: $\bar{\rho}_{f_1}, \bar{\rho}_{f_2}: G_\Q \rightarrow \GL_2(\calO_F/\varpi^m\calO_F)$ are isomorphic, where $\varpi$ is the uniformizer of $\calO_{F}$. For each prime $\ell \mid N_1 N_2$, let $v \mid \frakN_1 \frakN_2$ be the corresponding prime above $\ell$. Then the following congruence holds:
	\begin{align*}
		\prod_{\ell \mid N_1 N_2} \scrP_{\vbar}(f_1) \scrL_{\frakp}(f_1) \equiv \prod_{\ell \mid N_1 N_2} \scrP_{\vbar}(f_2)  \scrL_{\frakp}(f_2)  \pmod{\varpi^m \calW \llbracket \Gamma_K^{-} \rrbracket}.
	\end{align*}
	Moreover, one has the following:
	\begin{enumerate}
		\item $\mu(\scrL_{\frakp}(f_1)) = 0$ if and only if $\mu(\scrL_{\frakp}(f_2)) = 0$.
		\item Assuming that $\mu(\scrL_{\frakp}(f_1)) = \mu(\scrL_{\frakp}(f_2)) = 0$, 
		\[\sum_{\ell \mid N_1 N_2} \lambda(\scrP_{\vbar}(f_1)) + \lambda(\scrL_{\frakp}(f_1)) = \sum_{\ell \mid N_1 N_2} \lambda(\scrP_{\vbar}(f_2)) +  \lambda(\scrL_{\frakp}(f_2)).\]
	\end{enumerate}
\end{thm}

\begin{proof}
	Let $N^\sharp := \lcm_{\ell \mid N_1N_2}(N_1, N_2, \ell^2)$, and let $\frakN^\sharp = \lcm_{v \mid \frakN_1\frakN_2}(\frakN_1, \frakN_2, v^2)$. Since $f_1^{(N_1 N_2)} \equiv f_2^{(N_1 N_2)} \pmod{\varpi^m}$, Lemma \ref{lem:cong} gives the following congruence:
	\begin{align*}
		\scrL_{\frakp}(f_1^{(N_1 N_2)}) \equiv \scrL_{\frakp}(f_2^{(N_1 N_2)}) \pmod{\varpi^m \calW \llbracket \Gamma_K^{-} \rrbracket}.
	\end{align*}
	By repeatedly applying Theorem \ref{l-dep-2}, we have
	$$ \scrL_{\frakp}(f^{(N_1 N_2)})= \left( \prod_{\ell \mid N_1 N_2} \scrP_{\vbar}(f) \right) \scrL_{\frakp}(f)$$
	for each $f \in \{f_1, f_2\}$. Thus, the previous congruence is equivalent to
	\begin{align*}
		\left(\prod_{\ell \mid N_1 N_2} \scrP_{\vbar}(f_1) \right)\scrL_{\frakp}(f_1)  \equiv \left (\prod_{\ell \mid N_1 N_2}  \scrP_{\vbar}(f_2) \right) \scrL_{\frakp}(f_2)\\  \pmod{\varpi^m \calW \llbracket \Gamma_K^{-} \rrbracket}
	\end{align*}
	This congruence also holds over $\calW \llbracket \Gamma_K^- \rrbracket / \varpi \calW \llbracket \Gamma_K^- \rrbracket \simeq \Fbar_p \llbracket \Gamma_K^- \rrbracket \simeq \Fbar_p \llbracket T \rrbracket$.  Since $\mu(\scrP_{\ell}(f_1)) = \mu(\scrP_{\ell}(f_2)) = 0$ by Lemma \ref{lem:mu=0}, we have $\mu(\scrL_{\frakp}(f_1)) = 0$ if and only if $\mu(\scrL_{\frakp}(f_2)) = 0$.
	
	Note that for an element $\scrP \in \calW \llbracket \Gamma_K^{-} \llbracket \simeq \calW \rrbracket T \rrbracket$ with $\mu(\scrP) = 0$, we have $\lambda(\scrP) = \deg(\overline{\scrP})$, where $\overline{\scrP} \in \Fbar_p \llbracket T \rrbracket$ is the reduction of $\scrP$ $\mod{\varpi}$. When $\mu(\scrL_{\frakp}(f_1)) = \mu(\scrL_{\frakp}(f_2)) = 0$, it follows that 	
	$$\sum_{\ell \mid N_1 N_2} \lambda(\scrP_{\vbar}(f_1)) + \lambda(\scrL_{\frakp}(f_1)) = \sum_{\ell \mid N_1 N_2} \lambda(\scrP_{\vbar}(f_2)) + \lambda(\scrL_{\frakp}(f_2)).$$
\end{proof}

\section{Applications to generalized Heegner cycles} \label{sec:Heeg}
In this section, we follow the set-up of \cite[Section 4]{CH18}. As before, let $f \in S_{2r}(\Gamma_0(N))^{\new}$ be a normalized Hecke eigenform of weight $2r$ and level $N$ satisfying hypothesis \eqref{Hyp}.
\subsection{Generalized Heegner classes}  \label{sec:defHeeg}

Recall that $K = \Q(\sqrt{-D_K})$ where $D_K$ is the discriminant of $K$, and for $r > 1$ assume that either $-D_K > 3$ is odd, or $8 \mid D_K$. Such an assumption gurantees a canonical choice of elliptic curve $A$ with CM by $\calO_K$, defined over the real subfield of the Hilbert class field $H_K$ of $K$ \cite[Section 4.1]{CH18}. 

Recall that $V = V_f(r)$ is the self-dual $p$-adic Galois representation associated with $f$. Let $T$ be a $G_\Q$-stable lattice in $V$. For primes $p$ such that $p \nmid 2 (2r - 1)! N \phi(N)$, denote by $z_{f, \chi} \in H^1_f(K, T \otimes \chi)$ the generalized Heegner class attached to $(f, \chi)$ \cite[Section 4.5]{CH18}. We remark that the construction involves the aforementioned canonical CM elliptic curve $A$.

\subsection{Bloch-Kato logarithm map for a $p$-adic Galois representation}
Now, we recall the definition of the Bloch-Kato logarithm map. Let $\B_{\dR}, \B_{\cris}$ be Fontaine's rings of $p$-adic periods \cite[Definition 5.15, Definition 6.7]{FO08}, and let $t \in \B_{\dR}$ be Fontaine's $p$-adic analogue of $2\pi i$ \cite[Section 5.2.3]{FO08}.

Let $F/\Q_p$ be a finite extension. Suppose that $V$ is an $F$-vectorspace that is also a $G_L$-module for some finite extension $L/\Q_p$ (such as the $\Gal(\overline{\Q_p}/\Q_p)$-representation $V = V_f(r)$ where $F$ is a $p$-adic field containing the coefficients of $f$). 

Denote by $\D_{\dR,  L}(V)$ the filtered $(L \otimes_{\Q_p} F)$-module $(\B_{\dR} \otimes V)^{G_L}$ and define $H^1_f(L, V) := \ker(H^1(L, V) \rightarrow H^1(L, \B_{\cris} \otimes V))$ in accordance with \cite[(3.7.2)]{BK90}. If $V$ is a de Rham representation, the following exponential map is due to Bloch and Kato \cite[Section 3]{BK90}:
\[\exp: \frac{\D_{\dR, L}(V)}{\Fil^{0} \D_{\dR, L}(V)} \rightarrow  H^1_f(L, V).\]
The logarithm map is defined as its inverse:
\[\log: H^1_f(L, V) \rightarrow \frac{\D_{\dR, L}(V)}{\Fil^{0} \D_{\dR, L}(V)} = (\Fil^{0} \D_{\dR, L}(V^\ast(1)))^{\vee}.\]

In the special case where $V$ is the $p$-adic representation attached an abelian variety, $H^1_f(L, V)$ is the image of the Kummer map in $H^1(L, V)$ and the Bloch-Kato logarithm is the usual logarithm map (see \cite[Example 3.11]{BK90}).

For any $p$-adic field $L$ containing $H_{K, \frakp}$, there is a decomposition 
\[H^1_{\dR}(A/L) = H^0_{\dR}(A/L, \Omega^{1}_{A/L}) \oplus H^1_{\dR}(A/L, \Omega^{0}_{A/L}). 
\]
for an elliptic curve $A_{/L}$ with CM by $\calO_K$. Recall our fixed choice of N\'{e}ron differential $\omega_A$, and let $\eta_A \in H^1_{\dR}(A/L, \Omega^{0}_{A/L})$ such that $\langle \omega_A, \eta_A \rangle = 1$ under the algebraic deRham cup product. 

Let $\omega_A^{r - 1 + j} \eta_A^{r - 1 -j}$ be a basis of $\D_{\dR, F}(\Sym^{2r -2}H^1_{\et}(A_{\overline{\Q}}, \Q_p))$ as defined in \cite[(1.4.6)]{BDP13}.

Let $W_{2r - 2}$ be the Kuga-Sato variety of dimension $2r - 1$. To our cusp form $f$ with coefficients in a $p$-adic field $F$, one may attach an element in  $\widetilde\omega_f \in H^{2r - 1}(W_{2r - 2}/F)$ via \cite[(1.1.12)]{BDP13} and \cite[Lemma 2.2]{BDP13}. Moreover, $V_f$ may be realized as a quotient of $H^{2r - 1}_{\et}({W_{2r - 2}}_{/\overline{\Q}}, \overline{\Q}_p) \otimes_{\overline{\Q}_p} F$ by the work of Scholl \cite{Sch90}. Let $\omega_f \in \D(V_f)$ be the image of $\tilde{\omega}_f \in H_{\dR}^{2r - 1}(W_{2r - 2}/F)$ under the composition
\[
H_{\dR}^{2r - 1}(W_{2r - 2}/F) \simeq \D_{\dR}(H_{\et}^{2r - 1}({W_{2r - 2}}_{/\overline{\Q}}, \overline{\Q_p}) \otimes_{\Q_p} F) \rightarrow \D_{\dR}(V_f).
\]

\subsection{A $p$-adic Gross-Zagier formula}
Recall the following $p$-adic Gross-Zagier formula \cite[Theorem 4.9]{CH18}, with the constant term later corrected in the extension to the quaternionic setting due to Magrone \cite[Theorem 6.4]{Mag22}. We remark that Theorem 4.9 of \cite{CH18} extends the main result of Bertolini-Darmon-Prasanna (see \cite[p.1083]{BDP13}, \cite[Theorem 5.13]{BDP13}) to characters that are ramified at $p$. 

\begin{thm} \label{thm: Castella-Hsieh} \cite[Theorem 4.9]{CH18}
	Suppose $p=\frakp \overline{\frakp}$ splits in $K$ and let $f \in S_{2r}(\Gamma_0(N))^{\new}$ be a Hecke eigen-newform of weight $2r$. If $\chi = \widehat{\phi} \in \mathfrak{X}_{p^{\infty}}$ is the $p$-adic avatar of an anticyclotomic Hecke character of infinity type $(j, -j)$ with $-r < j < r$ and conductor $p^n \mathcal{O}_K$ with $n \geq 1$, then
	\[
	\frac{\scrL_{\frakp}(f)(\chi^{-1})}{\Omega_p^{-2j}} =\frac{\mathfrak{g}(\phi_{\frakp}^{-1}) (\sqrt{-D_K})^{r + j} p^{n(-j - r)} \chi_{\frakp}^{-1}(p^n)}{(r - 1 + j)!} \cdot \langle \log_{\frakp} (z_{f, \chi}), \omega_f \otimes \omega_{A}^{r - 1 + j} \eta_A^{r - 1 - j} t^{1 -2r} \rangle,
	\]
	where $\Omega_p$ is the $p$-adic period of the canonical elliptic curve $A$ in Section \ref{sec:defHeeg}.
\end{thm}

In a similar manner to \cite[Corollary 6.3]{Bur17}, we would like to understand the $p$-adic valuation of $\langle \log_{\frakp} (z_{f, \chi}), \omega_f \otimes \omega_{A}^{r - 1 + j} \eta_A^{r - 1 - j} t^{1 -2r} \rangle$. It follows from Theorem \ref{thm: Castella-Hsieh} that we have an inequality
\[v_p\left(\langle \log_{\frakp} (z_{f, \chi}), \omega_f \otimes \omega_{A}^{r - 1 + j} \eta_A^{r - 1 - j} t^{1 -2r} \rangle \right) \geq n \left (j + r - \frac{1}{2}  - v_p(\chi_\frakp^{-1}(p)) \right) \] %= n \left(r + j - \frac{1}{2}  v_p(\chi_\frakp^{-1}(p)) \right)\]
for every anticyclotomic Hecke character $\phi$ of conductor $p^n$ and infinity type $(j, - j)$ with $-r < j < r$. Here we used the fact that $v_p(\frakg(\phi_\frakp^{-1})) = n/2$ for $n \geq 0$. Under the  extra conditions
\[\begin{cases} 
	\text{the level $N$ is square free} \\
	\overline{\rho}_f \text{ is absolutely irreducible},\\ 
\end{cases}\]
the $\mu$-invariant $\mu(\scrL_{\frakp}(f))$ vanishes \cite[Theorem 5.7]{Bur17} (see also \cite[Theorem B]{Hsi14} for the same statement under slightly different hypotheses) and there is an asymptotic formula:
\[\liminf_{\widehat{\phi}^{-1} \in \mathfrak{X}_{p^{\infty}}}  v_p\left(\langle \log_{\frakp} (z_{f, \chi}), \omega_f \otimes \omega_{A}^{r - 1 + j} \eta_A^{r - 1 - j} t^{1 -2r} \rangle \right) - n \left (r + j - \frac{1}{2} -  v_p(\chi_\frakp^{-1}(p)) \right) = 0,\]
where $p^n$ is the conductor of $\phi$.

We now recall the set-up of Section \ref{sec:cong}. Let $f_1 \in S_{2r_1}(\Gamma_0(N_1))^\new, f_2 \in S_{2r_2}(\Gamma_0(N_2))^\new$ be normalized Hecke eigenforms satisfying hypothesis \eqref{Hyp}. Moreover, suppose that $\chi = \widehat{\phi} \in \mathfrak{X}_{p^{\infty}}$ is the $p$-adic avatar of an anticyclotomic Hecke character of infinity type $(j, -j)$ with $-r < j < r$ and conductor $p^n \mathcal{O}_K$ with $n \geq 1$. Let $F$ be a finite extension of $\Q_p$ containing the Hecke eigenvalues of $f_1$ and $f_2$ as well as the values of $\phi$, and let $\calW$ be the ring of integers of the compositum $F \cdot \widehat{\Q_p^{nr}}$. The following Theorem directly follows from Theorem \ref{thm:main} and Theorem \ref{thm: Castella-Hsieh}.
\begin{thm} Suppose that $f_1, f_2$ induce isomorphic semi-simplified $\mod{\varpi}^m$ Galois representations: $\bar{\rho}_{f_1}, \bar{\rho}_{f_2}: G_\Q \rightarrow \GL_2(\calO_F/\varpi^m\calO_F)$, where $\varpi$ is the uniformizer of $\calW$ and $\mu(\scrL_\frakp(f_1)) = \mu(\scrL_\frakp(f_2)) = 0$. Then 
	\begin{align*}
		v_p (\langle \prod_{\ell \mid N_1 N_2} \scrP_{\vbar}(f_1) (\chi^{-1}) \log_{\frakp} (z_{f_1, \chi}), \omega_f \otimes \omega_{A}^{r - 1 + j} \eta_A^{r - 1 - j} t^{1 -2r} \rangle \\ -  \langle \prod_{\ell \mid N_1 N_2} \scrP_{\vbar}(f_2) (\chi^{-1}) \log_{\frakp} (z_{f_2, \chi}),  \omega_f \otimes \omega_{A}^{r - 1 + j} \eta_A^{r - 1 - j} t^{1 -2r} \rangle) \\ \geq n \left(j + r - \frac{1}{2} - v_p(\chi_\frakp^{-1}(p)) \right) + v_p(\varpi^{m}),
	\end{align*}
	where $\scrP_{\overline{v}}(f_1), \scrP_{\overline{v}}(f_2)$ are defined in Definition \ref{defn:Euler}.
\end{thm}

Let $X_0(N)$ be the modular curve of level $\Gamma_0(N)$ and let $J_0(N)$ be its Jacobian. For a Hecke eigen-newform $f \in S_{2}(\Gamma_0(N))^{\new}$ of weight $2$ satisfying Hypothesis \eqref{Hyp} and a finite character $\chi$ of conductor $p^n$, define
\[P(\chi^{-1}) :=  \sum_{\fraka \in \Cl{\calO_{p^n}}}  \chi(\fraka) ([(A_\fraka, A_\fraka[\frakN])] - [\infty]) \in J_0(N) \otimes \C_p\]
and let $P_f(\chi^{-1}) := \pi_f(P(\chi^{-1}))$ under the modular parametrization $\pi_f: J_0(N) \rightarrow A_f$. Moroever, let $\omega_{A_f} \in H^0(A_f, {\Omega}^1_{A_f})$ be the differential induced by $f(q) \frac{dq}{q} \in H^0(X_0(N), {\Omega}^1_{X_0(N)})$ under the Abel-Jacobi map $\iota: X_0(N) \hookrightarrow J_0(N)$ and the projection $\pi_f$.

In the case of weight $2$ forms, we obtain the following extension of \cite[Theorem 3.9]{KL19}:

\newpage
\begin{thm} \label{main-3-cor} Suppose that $f_1 \in S_{2}(\Gamma_0(N_1))^\new, f_2 \in S_{2}(\Gamma_0(N_2))^\new$ induce isomorphic semi-simplified $\mod{\varpi}^m$ Galois representations: $\bar{\rho}_{f_1}, \bar{\rho}_{f_2}: G_\Q \rightarrow \GL_2(\calO_F/\varpi^m\calO_F)$, where $\varpi$ is a uniformizer of $F$. Then
	\begin{align*}
		v_p (\langle \prod_{\ell \mid N_1 N_2} \scrP_{\vbar}(f_1) (\chi^{-1}) \log_{\omega_{A_{f_1}}} (P_{f_1}({\chi^{-1}})) - \prod_{\ell \mid N_1 N_2} \scrP_{\vbar}(f_2) (\chi^{-1}) \log_{\omega_{A_{f_2}}} (P_{f_2}({\chi^{-1}})) \rangle) \\ \geq \frac{n}{2} + v_p(\varpi^{m}).
	\end{align*}
\end{thm}

\section{Applications to the Iwasawa Main Conjecture for the BDP Selmer group}
\label{sec:IMC}
In this section, we recall the definition of the Bertolini-Darmon-Prasanna (BDP) Selmer group and the corresponding Iwasawa Main Conjecture, which is equivalent to the Heegner Point Main Conjecture formulated by Perrin-Riou \cite{P-R87}. We will see that the Iwasawa Main Conjecture propagates in a family of modular forms with isomorphic semi-simplified residual representations.

We give the following definitions based on \cite[Definitions 2.1, 2.2]{Cas17}, \cite[p.98]{Gre99}. Assume that $f \in S_{2r}(\Gamma_0(N))^{\new}$ is ordinary at $p$, i.e. $a_p(f) \in \Z_p^\times$. Recall the $p$-adic representation $V = V_f(r)$ of $\Gal(\Qbar_p/\Q_p)$ attached to $f$. There exists a $\Gal(\Qbar_p/ \Q_p)$-stable filtration
$$0 \rightarrow \scrF^+ V \rightarrow V \rightarrow \scrF^- V \rightarrow 0$$

where $\scrF^+V$ and $\scrF^-V$ are both $1$-dimensional representations. Let $T$ be a $G_\Q$-stable lattice in $V$ and let $A = V/T$. We also define $\scrF^+ T = T \cap \scrF^+ V$, $\scrF^- T = T/ \scrF^+ T$, and $\scrF^+ A = \scrF^+ V/ \scrF^+ T$, $\scrF^- A = A/ \scrF^+ A$.  

To define the BDP Selmer group, we recall the following local conditions above $p$, where $M$ is $A, V$ or $T$. Let $L/K$ be a finite extension of number fields, and let $v$ be a prime of $L$. 

\begin{defn}
	The Greenberg local condition is defined as 
	$$H^1_{\Gr}(L_v, M) := \begin{cases}
		\ker \left(H^1(L_v, M) \rightarrow H^1(L_v^{nr}, \scrF^- M)\right) & \text{if } v \mid p,\\
		\ker \left(H^1(L_v, M) \rightarrow H^1(L_v^{nr}, M) \right) & \text{if otherwise}.
	\end{cases}$$
\end{defn}

\begin{defn}
	For $v \mid p$ and $\scrL_v \in \{\emptyset, \Gr, 0\}$, set
	$$H^1_{\scrL_v}(L_v, M):= \begin{cases}
		H^1(L_v, M) & \text{if } \scrL_v = \emptyset, \\
		H^1_{\Gr}(L_v, M) & \text{if } \scrL_v = \Gr, \\
		\{0\} & \text{if } \scrL_v = 0. 
	\end{cases}$$
\end{defn} 

Let $\Sigma$ be a finite set of primes of $K$ dividing the primes where $V$ is ramified as well as the primes dividing $p\infty$. We will denote by $L_\Sigma$ the maximal extension of $L$ unramified outside of the set of primes dividing the primes in $\Sigma$. 

\begin{defn} For a set of local conditions $\scrL  = \{\scrL_v\}_{v \mid p}$, we define
	\[\Sel_{\scrL} (L, M) = \ker \left(H^1(L_\Sigma/L, M) \rightarrow \prod_{v \nmid p} \frac{H^1(L_v, M)}{H^1_{\Gr}(L_v, M)} \times \prod_{v \mid p} \frac{H^1(L_v, M)}{H^1_{\scrL_v}(L_v, M)} \right).\]
\end{defn}

We abbreviate the Iwasawa algebras as $\Lambda := \Z_p \llbracket \Gamma_K^{-} \rrbracket$, $\Lambda^{\ur} := \calW \llbracket \Gamma_K^{-} \rrbracket$ and define $\mathbf{T} := T \otimes \Lambda$, and $\mathbf{A} := A \otimes \Lambda^\ast$, where $\Lambda^\ast$ is the Pontryagin dual of $\Lambda$.

Observe that there are isomorphisms $$\Sel_{\scrL}(K_\infty, \boldA) \simeq \varinjlim_{K \subset L \subset K_\infty} \Sel_{\scrL}(L, A), \text{          } \Sel_{\scrL}(K_\infty, \boldT) \simeq \varprojlim_{K \subset L \subset K_\infty} \Sel_{\scrL}(L, T)$$ with compatible local conditions $\scrL$. We will also denote by $X_{\scrL}(K, \boldA)$ the Pontryagin dual of $\Sel_{\scrL}(K, \boldA)$. 

The BDP Selmer group is defined as $\Sel_{\emptyset, 0}(K, \boldA)$, and we will denote its dual by $X_\frakp(K, \boldA): = X_{\emptyset, 0}(K, \boldA)$. One can formulate the following Iwasawa main conjecture \cite[Conjecture 2.4.7]{Cas13}:

\begin{conj}[Iwasawa main conjecture] \label{conj:IMC} The BDP Selmer group
	$X_{\frakp} (K, \mathbf{A})$ is $\Z_p \llbracket \Gamma_K^- \rrbracket$-cotorsion, and 
	$$\Char(X_{\frakp}(K, \mathbf{A})) \otimes_{\Z_p\llbracket \Gamma_K^{-} \rrbracket} \calW \llbracket \Gamma_K^{-} \rrbracket = (\scrL_{\frakp}(f)^2)$$
	as ideals in $\calW \llbracket \Gamma_K^{-} \rrbracket$, where $\Char(X_{\frakp}(K, \mathbf{A}))$ is the characteristic ideal of $X_{\frakp}(K, \mathbf{A})$. 
\end{conj}

%Let $\psi: \Gamma_K^{-} \rightarrow \calO_{\C_p}^\times$ be an anticyclotomic Hecke character of inifinity type $(-r, r)$, and $\Tw_{\psi}: \calW \llbracket  \widetilde{\Gamma} \rrbracket \xrightarrow{\sim} \calW \llbracket  \widetilde{\Gamma} \rrbracket$ be the twist on the Iwasawa algebra given by $\gamma \mapsto \psi(\gamma) \gamma$ for $\gamma \in \widetilde{\Gamma}$. Then $\scrL_{\frakp, \psi}(f) = \Tw_{\psi}(\scrL_{\frakp}(f))$. One can also define the twist $X_{\frakp, \psi} (K, \mathbf{A})$ to be the dual of $\Sel(K, A(\psi^{-1}) \otimes \Z_p \llbracket \Gamma_K^{-} \rrbracket^{\ast})$ where $A(\psi^{-1})$ corresponds to the representation $V \otimes \psi^{-1}$. The Iwasawa main conjecture \ref{conj:IMC} above is equivalent to the equality 
%\[\Char(X_{\frakp, \psi}(K, \mathbf{A})) = (\scrL_{\frakp, \psi}(f)^2)\] as ideals in $\calW \llbracket \Gamma_{K}^{-} \rrbracket$ \cite{}. 

We also remark that Conjecture \ref{conj:IMC} is equivalent to Perrin-Riou's Heegner Point Main Conjecture for forms corresponding to elliptic curves. For more details on this equivalence, we refer readers to \cite{BCK21, P-R87}.

Denote by $\mu_\anal(f)$ and $\lambda_\anal(f)$ the $\mu$ and $\lambda$-invariants of $\scrL_\frakp(f)$, respectively. Moreover, let $\mu_\alg(f) := \mu(X_\frakp(K, \boldA))$ and $\lambda_{\alg}(f) := \lambda(X_{\frakp}(K, \boldA))$ be the algebraic $\mu$ and $\lambda$-invariants of $X_{\frakp}(K, \boldA)$.  Combining our results with the work of Lei-Mueller-Xia \cite{LMX23}, we obtain the following:

\begin{thm} \label{thm:IMC}
	Let $f_1 \in S_{2r_1}(\Gamma_0(N_1))^{\new}$ be a newform that satisfies Conjecture \ref{conj:IMC} and assume $\mu_{\anal}(f_1) = \mu_{\alg}(f_1) = 0$. 
	Suppose that $f_2 \in S_{2r_2}(\Gamma_0(N_2))^{\new}$ is a newform that satisfies the divisibility
	\[\scrL_{\frakp}(f_2)^2 \in \Char_{\Lambda}(X_{\frakp}(K, \boldA_2)),\]
	where $X_\frakp(K, \boldA_i)$ is the dual BDP Selmer group for $f_i$, $i \in \{1, 2\}$. Further suppose that $\overline{\rho}_{f_1} \simeq \overline{\rho}_{f_2} \pmod{\varpi}$ and $H^0(K_w, A_i) = 0$ for every $w \mid p$ and $i \in \{1, 2\}$. Then $\mu_\anal(f_2) = \mu_{\alg} (f_2) = 0$ and Conjecture \ref{conj:IMC} also holds for $f_2$.
\end{thm}

\begin{proof}
	Under these hypotheses, Theorem \ref{thm:main} and \cite[Corollary 3.8]{LMX23} imply that $\mu_{\alg}(f_2) = \mu_{\anal}(f_2) = 0.$
	
	Moreover, we also have
	\begin{equation} \label{eqn:lambda}
		2 \lambda(\scrL_{\frakp}(f_1)) + 2 \sum_{\ell \mid N_1 N_2} \lambda(\scrP_{\vbar}(f_1)) = 2 \lambda(\scrL_{\frakp}(f_2)) + 2 \sum_{\ell \mid N_1 N_2} \lambda(\scrP_{\vbar}(f_2))
	\end{equation}
	for any splitting $\ell = v \overline{v}$ in $K$ of the primes $\ell \mid N_1 N_2$. For $i \in \{1, 2\}$, each $\scrP_{\overline{v}}(f_i)$ is defined in Definition \ref{defn:Euler}. 
	
	By \cite[Corollary 3.8]{LMX23}, one also has
	\begin{equation} \label{eqn:lambda'}
		\lambda(X_\frakp(K, \boldA_1)) + 2 \sum_{\ell \mid N_1 N_2} \lambda(\scrP_{\vbar}(f_1)) = \lambda(X_\frakp(K, \boldA_2)) +  2 \sum_{\ell \mid N_1 N_2} \lambda(\scrP_{\vbar}(f_2)).	
	\end{equation}		Conjecture \ref{conj:IMC} for $f_1$ gives $\lambda_\anal(f_1) = \lambda_\alg(f_1)$. The equalities (\ref{eqn:lambda}) and (\ref{eqn:lambda'}) together imply that $\lambda_{\anal}(f_2) = \lambda_{\alg}(f_2)$. Combined with the divisibility for $f_2$, we conclude that Conjecture \ref{conj:IMC} also holds for $f_2$. 
\end{proof}

%%===========================================================================================%%
%% If you are submitting to one of the Nature Portfolio journals, using the eJP submission   %%
%% system, please include the references within the manuscript file itself. You may do this  %%
%% by copying the reference list from your .bbl file, paste it into the main manuscript .tex %%
%% file, and delete the associated \verb+\bibliography+ commands.                            %%
%%===========================================================================================%%

\bibliography{sn-bibliography}% common bib file

%% BioMed_Central_Bib_Style_v1.01

\begin{thebibliography}{30}
% BibTex style file: bmc-mathphys.bst (version 2.1), 2014-07-24
\ifx \bisbn   \undefined \def \bisbn  #1{ISBN #1}\fi
\ifx \binits  \undefined \def \binits#1{#1}\fi
\ifx \bauthor  \undefined \def \bauthor#1{#1}\fi
\ifx \batitle  \undefined \def \batitle#1{#1}\fi
\ifx \bjtitle  \undefined \def \bjtitle#1{#1}\fi
\ifx \bvolume  \undefined \def \bvolume#1{\textbf{#1}}\fi
\ifx \byear  \undefined \def \byear#1{#1}\fi
\ifx \bissue  \undefined \def \bissue#1{#1}\fi
\ifx \bfpage  \undefined \def \bfpage#1{#1}\fi
\ifx \blpage  \undefined \def \blpage #1{#1}\fi
\ifx \burl  \undefined \def \burl#1{\textsf{#1}}\fi
\ifx \doiurl  \undefined \def \doiurl#1{\url{https://doi.org/#1}}\fi
\ifx \betal  \undefined \def \betal{\textit{et al.}}\fi
\ifx \binstitute  \undefined \def \binstitute#1{#1}\fi
\ifx \binstitutionaled  \undefined \def \binstitutionaled#1{#1}\fi
\ifx \bctitle  \undefined \def \bctitle#1{#1}\fi
\ifx \beditor  \undefined \def \beditor#1{#1}\fi
\ifx \bpublisher  \undefined \def \bpublisher#1{#1}\fi
\ifx \bbtitle  \undefined \def \bbtitle#1{#1}\fi
\ifx \bedition  \undefined \def \bedition#1{#1}\fi
\ifx \bseriesno  \undefined \def \bseriesno#1{#1}\fi
\ifx \blocation  \undefined \def \blocation#1{#1}\fi
\ifx \bsertitle  \undefined \def \bsertitle#1{#1}\fi
\ifx \bsnm \undefined \def \bsnm#1{#1}\fi
\ifx \bsuffix \undefined \def \bsuffix#1{#1}\fi
\ifx \bparticle \undefined \def \bparticle#1{#1}\fi
\ifx \barticle \undefined \def \barticle#1{#1}\fi
\bibcommenthead
\ifx \bconfdate \undefined \def \bconfdate #1{#1}\fi
\ifx \botherref \undefined \def \botherref #1{#1}\fi
\ifx \url \undefined \def \url#1{\textsf{#1}}\fi
\ifx \bchapter \undefined \def \bchapter#1{#1}\fi
\ifx \bbook \undefined \def \bbook#1{#1}\fi
\ifx \bcomment \undefined \def \bcomment#1{#1}\fi
\ifx \oauthor \undefined \def \oauthor#1{#1}\fi
\ifx \citeauthoryear \undefined \def \citeauthoryear#1{#1}\fi
\ifx \endbibitem  \undefined \def \endbibitem {}\fi
\ifx \bconflocation  \undefined \def \bconflocation#1{#1}\fi
\ifx \arxivurl  \undefined \def \arxivurl#1{\textsf{#1}}\fi
\csname PreBibitemsHook\endcsname

%%% 1
\bibitem[\protect\citeauthoryear{Greenberg and Vatsal}{2000}]{GV00}
\begin{barticle}
\bauthor{\bsnm{Greenberg}, \binits{R.}},
\bauthor{\bsnm{Vatsal}, \binits{V.}}:
\batitle{On the {I}wasawa invariants of elliptic curves}.
\bjtitle{Invent. Math.}
\bvolume{142}(\bissue{1}),
\bfpage{17}--\blpage{63}
(\byear{2000})
\doiurl{10.1007/s002220000080}
\end{barticle}
\endbibitem

%%% 2
\bibitem[\protect\citeauthoryear{Emerton et~al.}{2006}]{EPW06}
\begin{barticle}
\bauthor{\bsnm{Emerton}, \binits{M.}},
\bauthor{\bsnm{Pollack}, \binits{R.}},
\bauthor{\bsnm{Weston}, \binits{T.}}:
\batitle{Variation of {I}wasawa invariants in {H}ida families}.
\bjtitle{Invent. Math.}
\bvolume{163}(\bissue{3}),
\bfpage{523}--\blpage{580}
(\byear{2006})
\doiurl{10.1007/s00222-005-0467-7}
\end{barticle}
\endbibitem

%%% 3
\bibitem[\protect\citeauthoryear{Castella and Hsieh}{2018}]{CH18}
\begin{barticle}
\bauthor{\bsnm{Castella}, \binits{F.}},
\bauthor{\bsnm{Hsieh}, \binits{M.-L.}}:
\batitle{Heegner cycles and {$p$}-adic {$L$}-functions}.
\bjtitle{Math. Ann.}
\bvolume{370}(\bissue{1-2}),
\bfpage{567}--\blpage{628}
(\byear{2018})
\doiurl{10.1007/s00208-017-1517-3}
\end{barticle}
\endbibitem

%%% 4
\bibitem[\protect\citeauthoryear{Bertolini et~al.}{2013}]{BDP13}
\begin{barticle}
\bauthor{\bsnm{Bertolini}, \binits{M.}},
\bauthor{\bsnm{Darmon}, \binits{H.}},
\bauthor{\bsnm{Prasanna}, \binits{K.}}:
\batitle{Generalized {H}eegner cycles and {$p$}-adic {R}ankin {$L$}-series}.
\bjtitle{Duke Math. J.}
\bvolume{162}(\bissue{6}),
\bfpage{1033}--\blpage{1148}
(\byear{2013})
\doiurl{10.1215/00127094-2142056} .
\bcomment{With an appendix by Brian Conrad}
\end{barticle}
\endbibitem

%%% 5
\bibitem[\protect\citeauthoryear{Brako\v{c}evi\'{c}}{2011}]{Bra11}
\begin{botherref}
\oauthor{\bsnm{Brako\v{c}evi\'{c}}, \binits{M.}}:
Anticyclotomic {$p$}-adic {$L$}-function of central critical {R}ankin-{S}elberg
  {$L$}-value.
Int. Math. Res. Not. IMRN
(21),
4967--5018
(2011)
\doiurl{10.1093/imrn/rnq275}
\end{botherref}
\endbibitem

%%% 6
\bibitem[\protect\citeauthoryear{Pollack and Weston}{2011}]{PW11}
\begin{barticle}
\bauthor{\bsnm{Pollack}, \binits{R.}},
\bauthor{\bsnm{Weston}, \binits{T.}}:
\batitle{On anticyclotomic {$\mu$}-invariants of modular forms}.
\bjtitle{Compos. Math.}
\bvolume{147}(\bissue{5}),
\bfpage{1353}--\blpage{1381}
(\byear{2011})
\doiurl{10.1112/S0010437X11005318}
\end{barticle}
\endbibitem

%%% 7
\bibitem[\protect\citeauthoryear{Kim}{2017}]{Kim17}
\begin{barticle}
\bauthor{\bsnm{Kim}, \binits{C.-H.}}:
\batitle{Anticyclotomic {I}wasawa invariants and congruences of modular forms}.
\bjtitle{Asian J. Math.}
\bvolume{21}(\bissue{3}),
\bfpage{499}--\blpage{530}
(\byear{2017})
\doiurl{10.4310/AJM.2017.v21.n3.a5}
\end{barticle}
\endbibitem

%%% 8
\bibitem[\protect\citeauthoryear{Castella et~al.}{2017}]{CKL17}
\begin{barticle}
\bauthor{\bsnm{Castella}, \binits{F.}},
\bauthor{\bsnm{Kim}, \binits{C.-H.}},
\bauthor{\bsnm{Longo}, \binits{M.}}:
\batitle{Variation of anticyclotomic {I}wasawa invariants in {H}ida families}.
\bjtitle{Algebra Number Theory}
\bvolume{11}(\bissue{10}),
\bfpage{2339}--\blpage{2368}
(\byear{2017})
\doiurl{10.2140/ant.2017.11.2339}
\end{barticle}
\endbibitem

%%% 9
\bibitem[\protect\citeauthoryear{Lei et~al.}{2023}]{LMX23}
\begin{barticle}
\bauthor{\bsnm{Lei}, \binits{A.}},
\bauthor{\bsnm{Müller}, \binits{K.}},
\bauthor{\bsnm{Xia}, \binits{J.}}:
\batitle{On the iwasawa invariants of bdp selmer groups and bdp p-adic
  l-functions}.
\bjtitle{Forum Mathematicum}
(\byear{2023})
\doiurl{10.1515/forum-2023-0049}
\end{barticle}
\endbibitem

%%% 10
\bibitem[\protect\citeauthoryear{Kriz and Li}{2019}]{KL19}
\begin{barticle}
\bauthor{\bsnm{Kriz}, \binits{D.}},
\bauthor{\bsnm{Li}, \binits{C.}}:
\batitle{Goldfeld's conjecture and congruences between {H}eegner points}.
\bjtitle{Forum Math. Sigma}
\bvolume{7},
\bfpage{15}--\blpage{80}
(\byear{2019})
\doiurl{10.1017/fms.2019.9}
\end{barticle}
\endbibitem

%%% 11
\bibitem[\protect\citeauthoryear{Castella et~al.}{2022}]{CGLS22}
\begin{barticle}
\bauthor{\bsnm{Castella}, \binits{F.}},
\bauthor{\bsnm{Grossi}, \binits{G.}},
\bauthor{\bsnm{Lee}, \binits{J.}},
\bauthor{\bsnm{Skinner}, \binits{C.}}:
\batitle{On the anticyclotomic {I}wasawa theory of rational elliptic curves at
  {E}isenstein primes}.
\bjtitle{Invent. Math.}
\bvolume{227}(\bissue{2}),
\bfpage{517}--\blpage{580}
(\byear{2022})
\doiurl{10.1007/s00222-021-01072-y}
\end{barticle}
\endbibitem

%%% 12
\bibitem[\protect\citeauthoryear{Hida}{2004}]{Hid04}
\begin{bbook}
\bauthor{\bsnm{Hida}, \binits{H.}}:
\bbtitle{{$p$}-adic Automorphic Forms on {S}himura Varieties}.
\bsertitle{Springer Monographs in Mathematics},
p. \bfpage{390}.
\bpublisher{Springer}, \blocation{???}
(\byear{2004}).
\doiurl{10.1007/978-1-4684-9390-0} .
\burl{https://doi.org/10.1007/978-1-4684-9390-0}
\end{bbook}
\endbibitem

%%% 13
\bibitem[\protect\citeauthoryear{Katz}{1973}]{Kat73}
\begin{bchapter}
\bauthor{\bsnm{Katz}, \binits{N.M.}}:
\bctitle{{$p$}-adic properties of modular schemes and modular forms}.
In: \bbtitle{Modular Functions of One Variable, {III} ({P}roc. {I}nternat.
  {S}ummer {S}chool, {U}niv. {A}ntwerp, {A}ntwerp, 1972)}.
\bsertitle{Lecture Notes in Math},
vol. \bseriesno{Vol. 350},
pp. \bfpage{69}--\blpage{190}.
\bpublisher{Springer}, \blocation{???}
(\byear{1973})
\end{bchapter}
\endbibitem

%%% 14
\bibitem[\protect\citeauthoryear{Katz}{1978}]{Kat78}
\begin{barticle}
\bauthor{\bsnm{Katz}, \binits{N.M.}}:
\batitle{{$p$}-adic {$L$}-functions for {CM} fields}.
\bjtitle{Invent. Math.}
\bvolume{49}(\bissue{3}),
\bfpage{199}--\blpage{297}
(\byear{1978})
\doiurl{10.1007/BF01390187}
\end{barticle}
\endbibitem

%%% 15
\bibitem[\protect\citeauthoryear{Tate}{1967}]{Tat67}
\begin{bchapter}
\bauthor{\bsnm{Tate}, \binits{J.T.}}:
\bctitle{{$p$}-divisible groups}.
In: \bbtitle{Proc. {C}onf. {L}ocal {F}ields ({D}riebergen, 1966)},
pp. \bfpage{158}--\blpage{183}.
\bpublisher{Springer}, \blocation{???}
(\byear{1967})
\end{bchapter}
\endbibitem

%%% 16
\bibitem[\protect\citeauthoryear{Katz}{1981}]{Kat81}
\begin{bchapter}
\bauthor{\bsnm{Katz}, \binits{N.}}:
\bctitle{Serre-{T}ate local moduli}.
In: \bbtitle{Algebraic Surfaces ({O}rsay, 1976--78)}.
\bsertitle{Lecture Notes in Math},
vol. \bseriesno{868},
pp. \bfpage{138}--\blpage{202}.
\bpublisher{Springer}, \blocation{???}
(\byear{1981})
\end{bchapter}
\endbibitem

%%% 17
\bibitem[\protect\citeauthoryear{Hida}{1993}]{Hid93}
\begin{bbook}
\bauthor{\bsnm{Hida}, \binits{H.}}:
\bbtitle{Elementary Theory of {$L$}-functions and {E}isenstein Series}.
\bsertitle{London Mathematical Society Student Texts},
vol. \bseriesno{26},
p. \bfpage{386}.
\bpublisher{Cambridge University Press, Cambridge}, \blocation{???}
(\byear{1993}).
\doiurl{10.1017/CBO9780511623691} .
\burl{https://doi.org/10.1017/CBO9780511623691}
\end{bbook}
\endbibitem

%%% 18
\bibitem[\protect\citeauthoryear{Burungale et~al.}{2021}]{BCK21}
\begin{barticle}
\bauthor{\bsnm{Burungale}, \binits{A.}},
\bauthor{\bsnm{Castella}, \binits{F.}},
\bauthor{\bsnm{Kim}, \binits{C.-H.}}:
\batitle{A proof of {P}errin-{R}iou's {H}eegner point main conjecture}.
\bjtitle{Algebra Number Theory}
\bvolume{15}(\bissue{7}),
\bfpage{1627}--\blpage{1653}
(\byear{2021})
\doiurl{10.2140/ant.2021.15.1627}
\end{barticle}
\endbibitem

%%% 19
\bibitem[\protect\citeauthoryear{Hunter~Brooks}{2015}]{Bro15}
\begin{botherref}
\oauthor{\bsnm{Hunter~Brooks}, \binits{E.}}:
Shimura curves and special values of {$p$}-adic {$L$}-functions.
Int. Math. Res. Not. IMRN
(12),
4177--4241
(2015)
\doiurl{10.1093/imrn/rnu062}
\end{botherref}
\endbibitem

%%% 20
\bibitem[\protect\citeauthoryear{Vatsal}{1999}]{Vat99}
\begin{barticle}
\bauthor{\bsnm{Vatsal}, \binits{V.}}:
\batitle{Canonical periods and congruence formulae}.
\bjtitle{Duke Math. J.}
\bvolume{98}(\bissue{2}),
\bfpage{397}--\blpage{419}
(\byear{1999})
\doiurl{10.1215/S0012-7094-99-09811-3}
\end{barticle}
\endbibitem

%%% 21
\bibitem[\protect\citeauthoryear{Fontaine and Ouyang}{2008}]{FO08}
\begin{botherref}
\oauthor{\bsnm{Fontaine}, \binits{J.-M.}},
\oauthor{\bsnm{Ouyang}, \binits{Y.}}:
Theory of p-adic galois representations.
preprint
(2008)
\end{botherref}
\endbibitem

%%% 22
\bibitem[\protect\citeauthoryear{Bloch and Kato}{1990}]{BK90}
\begin{bchapter}
\bauthor{\bsnm{Bloch}, \binits{S.}},
\bauthor{\bsnm{Kato}, \binits{K.}}:
\bctitle{{$L$}-functions and {T}amagawa numbers of motives}.
In: \bbtitle{The {G}rothendieck {F}estschrift, {V}ol. {I}}.
\bsertitle{Progr. Math.},
vol. \bseriesno{86},
pp. \bfpage{333}--\blpage{400}.
\bpublisher{Birkh\"{a}user Boston, Boston, MA}, \blocation{???}
(\byear{1990})
\end{bchapter}
\endbibitem

%%% 23
\bibitem[\protect\citeauthoryear{Scholl}{1990}]{Sch90}
\begin{barticle}
\bauthor{\bsnm{Scholl}, \binits{A.J.}}:
\batitle{Motives for modular forms}.
\bjtitle{Invent. Math.}
\bvolume{100}(\bissue{2}),
\bfpage{419}--\blpage{430}
(\byear{1990})
\doiurl{10.1007/BF01231194}
\end{barticle}
\endbibitem

%%% 24
\bibitem[\protect\citeauthoryear{Magrone}{2022}]{Mag22}
\begin{barticle}
\bauthor{\bsnm{Magrone}, \binits{P.}}:
\batitle{Generalized {H}eegner cycles and {$p$}-adic {$L$}-functions in a
  quaternionic setting}.
\bjtitle{Ann. Sc. Norm. Super. Pisa Cl. Sci. (5)}
\bvolume{23}(\bissue{4}),
\bfpage{1807}--\blpage{1870}
(\byear{2022})
\end{barticle}
\endbibitem

%%% 25
\bibitem[\protect\citeauthoryear{Burungale}{2017}]{Bur17}
\begin{barticle}
\bauthor{\bsnm{Burungale}, \binits{A.A.}}:
\batitle{On the non-triviality of the {$p$}-adic {A}bel-{J}acobi image of
  generalised {H}eegner cycles modulo {$p$}, {II}: {S}himura curves}.
\bjtitle{J. Inst. Math. Jussieu}
\bvolume{16}(\bissue{1}),
\bfpage{189}--\blpage{222}
(\byear{2017})
\doiurl{10.1017/S147474801500016X}
\end{barticle}
\endbibitem

%%% 26
\bibitem[\protect\citeauthoryear{Hsieh}{2014}]{Hsi14}
\begin{barticle}
\bauthor{\bsnm{Hsieh}, \binits{M.-L.}}:
\batitle{Special values of anticyclotomic {R}ankin-{S}elberg {$L$}-functions}.
\bjtitle{Doc. Math.}
\bvolume{19},
\bfpage{709}--\blpage{767}
(\byear{2014})
\doiurl{10.1016/j.cnsns.2013.07.005}
\end{barticle}
\endbibitem

%%% 27
\bibitem[\protect\citeauthoryear{Perrin-Riou}{1987}]{P-R87}
\begin{barticle}
\bauthor{\bsnm{Perrin-Riou}, \binits{B.}}:
\batitle{Fonctions {$L$} {$p$}-adiques, th\'eorie d'{I}wasawa et points de
  {H}eegner}.
\bjtitle{Bull. Soc. Math. France}
\bvolume{115}(\bissue{4}),
\bfpage{399}--\blpage{456}
(\byear{1987})
\end{barticle}
\endbibitem

%%% 28
\bibitem[\protect\citeauthoryear{Castella}{2017}]{Cas17}
\begin{barticle}
\bauthor{\bsnm{Castella}, \binits{F.}}:
\batitle{{$p$}-adic heights of {H}eegner points and {B}eilinson-{F}lach
  classes}.
\bjtitle{J. Lond. Math. Soc. (2)}
\bvolume{96}(\bissue{1}),
\bfpage{156}--\blpage{180}
(\byear{2017})
\doiurl{10.1112/jlms.12058}
\end{barticle}
\endbibitem

%%% 29
\bibitem[\protect\citeauthoryear{Greenberg}{1999}]{Gre99}
\begin{bchapter}
\bauthor{\bsnm{Greenberg}, \binits{R.}}:
\bctitle{Iwasawa theory for elliptic curves}.
In: \bbtitle{Arithmetic Theory of Elliptic Curves ({C}etraro, 1997)}.
\bsertitle{Lecture Notes in Math.},
vol. \bseriesno{1716},
pp. \bfpage{51}--\blpage{144}.
\bpublisher{Springer}, \blocation{???}
(\byear{1999}).
\doiurl{10.1007/BFb0093453} .
\burl{https://doi.org/10.1007/BFb0093453}
\end{bchapter}
\endbibitem

%%% 30
\bibitem[\protect\citeauthoryear{Castella}{2013}]{Cas13}
\begin{botherref}
\oauthor{\bsnm{Castella}, \binits{F.}}:
On the p-adic variation of heegner points.
PhD thesis,
McGill University
(2013)
\end{botherref}
\endbibitem

\end{thebibliography}
%% if required, the content of .bbl file can be included here once bbl is generated
%%\input sn-article.bbl

\end{document}